\documentclass{interact}

\usepackage[T1]{fontenc}
\usepackage{lmodern}
\usepackage{mathtools,array,longtable}
\usepackage[numbers,sort&compress]{natbib}
\bibpunct[, ]{[}{]}{,}{n}{,}{,}

\usepackage[hidelinks]{hyperref}
\hypersetup{
  pdftitle={Eigenspaces of Preserved Graph Laplacian Eigenvalues under Positive Edge-Weight Addition},
  pdfauthor={Maribel Bueno, Susana Furtado and Javier Perez},
  pdfsubject={Graph Laplacian eigenvalues under positive edge updates},
  pdfkeywords={graph Laplacian, positive edge updates, eigenvalue preservation, eigenspace reconstruction, low-rank perturbations, exceptional weights}
}

\theoremstyle{plain}
\newtheorem{theorem}{Theorem}[section]
\newtheorem{proposition}[theorem]{Proposition}
\newtheorem{lemma}[theorem]{Lemma}
\newtheorem{corollary}[theorem]{Corollary}
\newtheorem{example}[theorem]{Example}
\theoremstyle{definition}
\newtheorem{definition}[theorem]{Definition}
\theoremstyle{remark}
\newtheorem{remark}[theorem]{Remark}

\DeclareMathOperator{\rank}{rank}
\DeclareMathOperator{\range}{range}

\begin{document}
\title{Eigenspaces of Preserved Graph Laplacian Eigenvalues under Positive Edge-Weight Addition}
\author{%
\name{Maribel Bueno\textsuperscript{a}\thanks{CONTACT Maribel Bueno. Email: mbueno@ucsb.edu},
Susana Furtado\textsuperscript{b} and Javier P\'erez\textsuperscript{c}}
\affil{\textsuperscript{a}Department of Mathematics and College of Creative Studies,
University of California, Santa Barbara, Santa Barbara, CA 93106, USA;
\textsuperscript{b}CEMS.UL and Faculdade de Economia, Universidade do Porto,
Rua Dr. Roberto Frias, 4200--464 Porto, Portugal;
\textsuperscript{c}Department of Mathematical Sciences, University of Montana,
Missoula, MT 59812, USA}
}

\maketitle

\noindent{\small\textbf{Approximate word count:} 7,530 (abstract and
main text, excluding  references and appendices).}

\par\bigskip

\begin{abstract}
We study how the eigenspace associated with a fixed graph Laplacian eigenvalue changes under positive edge updates, allowing both new edges and increases in existing edge weights. We show that a preserved eigenvalue may have inherited eigenvectors from the original eigenspace and emergent eigenvectors outside it. We characterize the inherited eigenspace and show that it is independent of the update weights. We characterize its orthogonal complement within the preserved eigenspace through a symmetric matrix whose size is at most the number of updated edges, yielding a multiplicity formula and an explicit reconstruction of the updated eigenspace. For fixed updated edges, we characterize when emergence is possible and show that, whenever it is, equal update weights suffice. The exceptional weights form the zero set of a nonzero multi-affine polynomial in reciprocal variables and have Lebesgue measure zero. Finally, emergence cannot occur at either of the two smallest Laplacian eigenvalues, counting multiplicities.

\end{abstract}


\begin{amscode}
Primary 05C50; Secondary 15A18, 15A10
\end{amscode}

\section{Introduction}

The spectrum of the Laplacian of a  weighted undirected graph
contains important information about the structure of the graph and plays a central
role in spectral graph theory and its applications; see, for example,
\cite{BrouwerHaemers2012,Chung1997,Merris1994}.

Adding new  positively-weighted edges to the graph,  increasing the weights of
existing edges, or doing both, changes the Laplacian by a positive
semidefinite perturbation whose rank is at most the number of updated
edges. We refer to these operations as \emph{positive edge updates}.
The effect of such  edge modifications
on Laplacian spectra has been studied from several points of view. Merris's
eigenvector principles give conditions under which certain Laplacian
eigenvalues and eigenvectors remain unchanged after edges are modified
\cite{Merris1998}. Generic spectral properties under changes in the
weights of existing edges were studied in \cite{PoignardPereiraPade2018}.
Other work  on edge modifications includes rank-one perturbation
methods for single-edge addition \cite{So1999}, graph transformations that
preserve an eigenvalue or change it in a prescribed way
\cite{CaputoKnippel2023}, edge perturbations for graphs with
additional structure \cite{CardosoRojo2017}, spectral integral variation under
edge addition \cite{BarikPati2005,Fan2002,Kirkland2004,WangFan2006}, spectral
$w$-variation {under a positive single-edge update} for trees and unicyclic
graphs \cite{BasumataryKalita2025,BasumataryKalita2026}, among others.

Merris's principle gives a simple condition under which 
an eigenvalue and a corresponding eigenvector remain unchanged after addition
of an edge. Suppose that a new unweighted (or, equivalently, with
unit weight)  edge joins vertices $i$ and $j$, and let $x$ be an
eigenvector of the original Laplacian. If $x_{i}=x_{j}$, then adding
that  edge preserves both $x$ and its eigenvalue,
regardless of the positive weight assigned to the edge. The same idea extends
directly to several edge updates: { if an original eigenvector has equal
entries at the endpoints of every updated edge,} then it remains an
eigenvector with the same eigenvalue. We present this
extension in Proposition~\ref{lem:inherited-eigenvectors}.

Simple examples show that when an eigenvalue is preserved under edge addition, eigenvectors not in the original eigenspace may occur.
This led us to a broader question. 
When an eigenvalue is preserved under edge addition, or more generally under a positive edge update, which corresponding eigenvectors are preserved and which are new?
A positive edge update increases the weights of one or more edges.
Adding a new edge can be viewed as increasing its weight from zero to a positive value.

For a preserved eigenvalue, we call an updated eigenvector \emph{inherited} if it belongs to the original eigenspace for the same eigenvalue, and \emph{emergent} otherwise. A preserved eigenvalue may have inherited eigenvectors, emergent eigenvectors, or both.

There is an important difference between these two mechanisms.
 For fixed updated edges, whether an original eigenvector is preserved
with the same eigenvalue does not depend on the positive update weights,
whereas emergence can occur only for special choices of those weights.
Example~\ref{ex:motivating} illustrates this behavior, and later we make
the meaning of ``special'' precise.

The goal of this paper is to understand these two mechanisms under
positive edge updates.
For a preserved eigenvalue, we express the updated eigenspace as the direct sum of the inherited subspace and a canonical emergent subspace. The inherited subspace is independent of the new edge weights, while the emergent subspace may change as the weights vary. We characterize this dependence through a smaller matrix in which the role of the edge weights can be analyzed explicitly. This formulation describes the complete updated eigenspace and allows us to determine whether emergence occurs, identify the positive weight vectors for which it occurs, and compute  both the dimension of the canonical emergent subspace and the multiplicity of the preserved eigenvalue.

Our problem is also related to the classical theory of low-rank perturbations of symmetric and Hermitian matrices~\cite{ArbenzGanderGolub1988,
ArbenzGolub1988,BeattieFox1988,BunchNielsenSorensen1978,Golub1973,
Haemers1995}. 
In our setting, the perturbation has additional structure coming from the updated edges and their weights. 
We exploit this structure to study the eigenspace associated with a preserved eigenvalue and, in particular, to distinguish inherited from emergent eigenvectors.
This leads to a reduced matrix whose size is at most the number of updated edges and which makes explicit the dependence on the individual update weights.

The structure of the paper and its main contributions are
summarized next. After introducing the update model and
motivating examples in Sections~\ref{sec:preliminaries} and~\ref{sec:mechanisms}, we
characterize the inherited eigenspace for a preserved eigenvalue
$\lambda$  in Section~\ref{sec:inherited-eigenspace}. We show that an original
eigenvector is inherited precisely when its entries agree at the
endpoints of every updated edge. We compute the dimension of the inherited eigenspace and show that both the eigenspace and its dimension depend only on the updated edges, and not on their positive weights.
 We also show that emergence at $\lambda$ is impossible if every eigenspace corresponding to an eigenvalue smaller than $\lambda$ is preserved. In particular, when the second
smallest eigenvalue, counting multiplicities, remains in the
updated spectrum, its updated eigenspace consists entirely of inherited vectors. For a connected graph, this applies  to its original algebraic
connectivity.

Emergent eigenvectors (together with the zero vector) need
not form a subspace. Moreover, two emergent eigenvectors may
differ by an inherited vector. In Section \ref{sec:inherited-emergent-decomposition},
we therefore define the canonical emergent subspace as the orthogonal
complement of the inherited subspace within the updated eigenspace.
 We show that every edge-difference vector arising from an updated eigenvector determines a unique representative in the canonical emergent subspace.

In Section~\ref{sec:reduced-characterization-reconstruction},  we
characterize the canonical emergent subspace through the nullspace of
a symmetric matrix whose size is at most the number of updated edges.
Its nullity equals the dimension of the canonical emergent subspace.
This gives a formula for the multiplicity of the preserved eigenvalue
after the update, together with an explicit reconstruction of the
canonical emergent subspace and, consequently, of the complete updated
eigenspace.
 Once the required data from the original Laplacian are available, the reduced matrix determines whether emergence occurs and the multiplicity of the preserved eigenvalue for each new choice of update weights, without solving a new full-size eigenvalue problem.
This is particularly useful when the number of updated edges is much
smaller than the number of vertices.

 In Section~\ref{sec:exceptional-weights}, we study the positive weight
vectors that produce emergence for a fixed original eigenvalue and a
fixed set of updated edges.
We determine when they
exist and show that, in reciprocal weight variables, they form the zero set,
within the positive weight space, of a nonzero polynomial that is
affine in each variable.
Consequently, the exceptional set has Lebesgue measure zero. This 
motivates calling them  \emph{exceptional}.
Although infinitely many exceptional weight vectors may exist, emergence is not generic. A small generic change in the weights will typically destroy emergence. 

The same polynomial also controls the number of emergent directions. We show that,
at a regular point of its zero set, the canonical emergent subspace has dimension one. 
Higher-dimensional emergence can occur only at singular points.
Along any fixed ray in the positive weight space, only finitely many positive scaling factors can be exceptional.
For a single updated edge, there is at most one exceptional positive weight.

The examples throughout the paper illustrate both preservation
mechanisms,  including emergence caused by adding a new edge or
increasing the weight of an existing edge.
The final example considers a cycle on \(1000\) vertices with three
updated edges.
A \(3\times3\) reduced matrix determines the exceptional weights and
the multiplicity of the preserved eigenvalue.
In particular, it identifies a weight choice that produces a
two-dimensional canonical emergent subspace, without computing the
relevant eigenspace of the full \(1000\times1000\) updated Laplacian.

The paper requires a fair amount of notation. 
For convenience, Appendix~\ref{app:notation} contains a summary of the notation used throughout the paper.

\section{Notation and preliminaries}
\label{sec:preliminaries}

We begin by introducing some notation used throughout the paper. We denote the
identity matrix of size $k$ by $I_k$, or simply by $I$ when its size is clear
from the context. We denote by $e_i$ the $i$th standard basis vector and by
$\mathbf 1_k$ the all-ones vector in $\mathbb R^k$. For a matrix $M$, we
denote its kernel, range, and rank by $\ker(M)$, $\range(M)$, and $\rank(M)$,
respectively. For a set of vectors $T\subseteq \mathbb{R}^k$, we denote by $\operatorname{span}(T)$  the subspace of $\mathbb{R}^k$ generated by $T$. If $\mathcal S$ is a subspace of $\mathbb R^k$, then
$\mathcal S^\perp$ denotes its orthogonal complement. All orthogonality is
with respect to the standard Euclidean inner product. The Moore--Penrose
inverse of $M$ is denoted by $M^\dagger$.

For a square matrix $M\in\mathbb R^{k\times k}$ and $\lambda\in\mathbb R$,
we define
\begin{equation}
\label{eq:eigenspace-notation}
E_\lambda(M):=\ker(M-\lambda I_k),
\qquad
m_\lambda(M):=\dim E_\lambda(M).
\end{equation}
We denote the spectrum of $M$ by $\sigma(M)$. If
$\lambda\notin\sigma(M)$, then $E_\lambda(M)=\{0\}$ and
$m_\lambda(M)=0$. If $\lambda\in\sigma(M)$, then $E_\lambda(M)$ is the
$\lambda$-eigenspace of $M$, and $m_\lambda(M)$ is the geometric
multiplicity of $\lambda$. When $M$ is symmetric, the geometric and
algebraic multiplicities coincide. A nonzero vector in $E_\lambda(M)$ is
called a $\lambda$-eigenvector of $M$.

\subsection{Positively weighted graphs}\label{sec:positivegraph}

Throughout this paper, all graphs are finite, simple, undirected, and
positively weighted. Let $G=(V,E,\eta)$ be such a graph, with vertex set
$V=\{1,\ldots,n\}$, edge set $E$, and weight function
$\eta:E\longrightarrow\mathbb R_{>0}$. For $i,j\in V$, define
\[
\eta_{ij}
:=
\begin{cases}
\eta(\{i,j\}),&\text{if }\{i,j\}\in E,\\
0,&\text{otherwise}.
\end{cases}
\]
Since $G$ is undirected, $\eta_{ij}=\eta_{ji}$. 

The \emph{weighted adjacency matrix} of $G$ is the symmetric matrix
$A_G=(\eta_{ij})_{i,j=1}^n$. The \emph{weighted degree} of vertex $i$ is
$d_i:=\sum_{j=1}^n\eta_{ij}$, and the \emph{weighted degree matrix} is
$D_G:=\operatorname{diag}(d_1,\ldots,d_n)$. The Laplacian of $G$ is
$L_G:=D_G-A_G$. Equivalently,
\begin{equation}
\label{eq:weighted-laplacian-sum}
L_G
=
\sum_{\{i,j\}\in E}
\eta_{ij}(e_i-e_j)(e_i-e_j)^T=
\frac{1}{2}
\sum_{i,j=1}^n
\eta_{ij}(e_i-e_j)(e_i-e_j)^T.
\end{equation}
Consequently, for every $z\in\mathbb R^n$,
\begin{equation}
\label{eq:weighted-laplacian-quadratic-form}
z^TL_Gz
=
\sum_{\{i,j\}\in E}\eta_{ij}(z_i-z_j)^2
\geq0.
\end{equation}
Thus, $L_G$ is a real symmetric positive semidefinite matrix. Moreover,
$L_G\mathbf 1_n=0$, and $\ker(L_G)$ consists of the vectors that are constant
on each connected component of $G$. In particular, $G$ is connected if and
only if
\[
\ker(L_G)=\operatorname{span}\{\mathbf 1_n\}.
\]

Fix $\lambda\in\sigma(L_G)$, and set $m:=m_\lambda(L_G)$. Throughout the paper, we denote by 
$Q_\lambda\in\mathbb R^{n\times m}$ the matrix with orthonormal columns spanning
$E_\lambda(L_G)$. Then
\begin{equation}
\label{eq:spectral-projectors}
P_\lambda:=Q_\lambda Q_\lambda^T
\quad \text{and} \quad
P_{\lambda,\perp}:=I_n-P_\lambda
\end{equation}
are the orthogonal projection matrices onto $E_\lambda(L_G)$ and its orthogonal complement
$E_\lambda(L_G)^\perp$, respectively.

Since \(L_G\) is symmetric, its eigenspaces corresponding to distinct
eigenvalues are mutually orthogonal and
\[
\mathbb R^n
=
\bigoplus_{\mu\in\sigma(L_G)} E_\mu(L_G).
\]
Then,
\begin{equation}
\label{eq:spectral-decomposition-L}
L_G
=
\sum_{\mu\in\sigma(L_G)} \mu P_\mu,
\qquad
I_n
=
\sum_{\mu\in\sigma(L_G)} P_\mu.
\end{equation}
The first identity is the \emph{spectral decomposition of \(L_G\)}.

\subsection{The positive edge-update model}
Let \(G=(V,E,\eta)\) be a graph as described in
Subsection~\ref{sec:positivegraph}. We modify \(G\) by increasing the
weights of some existing edges, adding new edges, or both. We encode such
an update by a graph \(H=(V,F,\omega)\) of the same type as \(G\), where
\[
F=\bigl\{f_\ell=\{i_\ell,j_\ell\}:\ell=1,\ldots,p\bigr\}
\]
is the set of updated edges and
$\omega:F\longrightarrow\mathbb R_{>0}$ assigns a positive weight 
to each  edge in $F$. We do not require $E$ and $F$ to be disjoint. 

For $i,j\in V$, define
\[
\omega_{ij}
:=
\begin{cases}
\omega(\{i,j\}),&\text{if }\{i,j\}\in F,\\
0,&\text{otherwise}.
\end{cases}
\]
  Corresponding to the indexing \(f_1,\ldots,f_p\) of the edges in \(F\),
we define the positive diagonal matrix
\begin{equation}
\label{eq:Omega-definition}
\Omega
:=
\operatorname{diag}
\bigl(\omega_{i_1j_1},\ldots,\omega_{i_pj_p}\bigr)
\in\mathbb R^{p\times p}.
\end{equation}
Its diagonal is  the \emph{associated vector of weights}. 
Define the updated weight function
\(\eta':E\cup F\longrightarrow\mathbb R_{>0}\) by
\[
\eta'(\{i,j\})
:=
\eta_{ij}+\omega_{ij},
\qquad
\{i,j\}\in E\cup F.
\]
The resulting graph is \(G+H:=(V,E\cup F,\eta')\).
 Thus, if
$\{i,j\}\in F\setminus E$, a new edge is added with weight $\omega_{ij}$,
whereas if $\{i,j\}\in E\cap F$, the weight of the existing edge is
increased by $\omega_{ij}$. Recall that we refer to either operation as a \emph{positive
edge update}.

For every update edge \(f_\ell=\{i_\ell,j_\ell\}\), choose
an ordering of its endpoints and define
\(b_\ell:=e_{i_\ell}-e_{j_\ell}\).
 Let
\begin{equation}
\label{eq:B-X-definitions}
B:=\begin{bmatrix}b_1&\cdots&b_p\end{bmatrix}
\in\mathbb R^{n\times p},
\qquad
X_{\Omega}:=B\Omega^{1/2}\in\mathbb R^{n\times p}.
\end{equation}

Thus, $B$ is an oriented incidence matrix of the underlying update graph
$H$, while $X_\Omega$ is its weighted incidence matrix. The $\ell$th
column of $X_\Omega$ is
$x_{\ell,\Omega}=\sqrt{\omega_{i_\ell j_\ell}}
b_{\ell}$, which is the weighted version of $b_{\ell}$. 

The Laplacian of the update graph $H$ is
\begin{equation}
\label{eq:update-laplacian}
L_H
=
B\Omega B^T
=
X_\Omega X_\Omega^T
=
\sum_{\ell=1}^p
\omega_{i_\ell j_\ell}
(e_{i_\ell}-e_{j_\ell})(e_{i_\ell}-e_{j_\ell})^T.
\end{equation}
Note that the choice of orientations does not affect the update: reversing the
orientation of an update edge changes the sign of the corresponding
column of \(B\) and of \(X_\Omega\), but leaves
\(X_\Omega X_\Omega^T\) unchanged.

Now the Laplacian of the updated graph $G+H$ is
\begin{equation}
\label{eq:updated-laplacian}
L_{G+H}
=
L_G+L_H
=
L_G+B\Omega B^T
=
L_G+X_\Omega X_\Omega^T.
\end{equation}
Thus, an edge update is a positive semidefinite perturbation of the
original Laplacian, with rank at most the number $p$ of updated edges.

We say that an eigenvalue $\lambda\in\sigma(L_G)$ is \emph{preserved} under
the positive edge update $L_H$  if $\lambda\in\sigma(L_{G+H})$. We define 
\begin{equation}
\label{def:Ulambda}
\mathcal U_\lambda
:=
E_\lambda(L_G)\cap E_\lambda(L_{G+H}).
\end{equation}
This subspace consists of the \(\lambda\)-eigenvectors of \(L_G\) that
remain \(\lambda\)-eigenvectors after the update, together with the zero
vector. In particular, \(\mathcal U_\lambda=\{0\}\) if
\(\lambda\notin\sigma(L_{G+H})\). The converse, however, need not hold:
the eigenvalue $\lambda$ may be preserved even when
$\mathcal U_\lambda=\{0\}$.

\section{Two mechanisms for eigenvalue preservation under positive edge updates} \label{sec:mechanisms}Fix $\lambda\in\sigma(L_G)$.
Our goal is to determine whether $\lambda$ is preserved under the update and, when it is, to describe the entire updated eigenspace $E_\lambda(L_{G+H})$ and understand how it relates to the original eigenspace $E_\lambda(L_G)$.

The eigenvalue $\lambda$ can be preserved through two different mechanisms. The first is \emph{inheritance}:  there is a nonzero vector $y\in E_\lambda(L_{G+H})$ that  is also a $\lambda$-eigenvector of $L_G$. Such a vector is called an \emph{inherited eigenvector}. The second mechanism is \emph{emergence}: the updated Laplacian has a $\lambda$-eigenvector that does not belong to the original $\lambda$-eigenspace. Thus, a nonzero vector $y\in E_\lambda(L_{G+H})$ is called an \emph{emergent eigenvector} if $y\notin E_\lambda(L_G)$. Every $\lambda$-eigenvector of $L_{G+H}$ is therefore either inherited or emergent, and $\lambda$ may be preserved through inheritance, emergence, or both mechanisms simultaneously. 

 The following examples illustrate that inheritance and emergence can occur at the same eigenvalue. 
They also show that emergence may happen only for particular choices of the update weights.
Later, we characterize these ``exceptional'' weights.

\begin{example}[Weight-dependent emergence after adding a new edge]\label{ex:motivating} Consider the positively weighted graph $G$ with vertex set $\{1,2,3,4,5\}$, edge set \[ E = \bigl\{ \{1,2\},\{1,4\},\{1,5\},\{2,3\} \bigr\}, \] and unit edge weights. Its Laplacian is 
\[ L_G = \begin{bmatrix} 3&-1&0&-1&-1\\ -1&2&-1&0&0\\ 0&-1&1&0&0\\ -1&0&0&1&0\\ -1&0&0&0&1 \end{bmatrix}. \] 
The eigenvalue $\lambda=1$ is simple, with $E_1(L_G)=\operatorname{span}\{y\}$, where \[ y= \begin{bmatrix} 0\\0\\0\\1\\-1 \end{bmatrix}. \] 
Let \(H\) be the graph with the same vertex set as \(G\), edge set
\(\{\{1,3\}\}\), and edge weight \(\omega>0\).  Let $b=e_1-e_3$, $\Omega=[\omega]$, and $X_{\Omega}=\sqrt{\omega}\,b$. Then 
\[ L_{G+H}= L_G+\omega bb^T = \begin{bmatrix} 3+\omega&-1&-\omega&-1&-1\\ -1&2&-1&0&0\\ -\omega&-1&1+\omega&0&0\\ -1&0&0&1&0\\ -1&0&0&0&1 \end{bmatrix}. \] 
Since $b^Ty=0$, we have $L_{G+H} y=L_Gy=y$. Thus, $y$ is an inherited $1$-eigenvector of $L_{G+H}$ for every $\omega>0$. Now consider \[ \widetilde y = \begin{bmatrix} 0\\-1\\-1\\1\\1 \end{bmatrix}. \] A direct calculation gives $L_G\widetilde y=\widetilde y-b$ and $b^T\widetilde y=1$. Consequently, $L_{G+H}\widetilde y=\widetilde y+(\omega-1)b$, and therefore $L_{G+H}\widetilde y=\widetilde y$ if and only if $\omega=1$. In fact, we can show that 
\[ E_1(L_{G+H}) = \begin{cases} \operatorname{span}\{y\},&\omega\neq1,\\[1mm] \operatorname{span}\{y,\widetilde y\},&\omega=1. \end{cases} \] Thus, inheritance occurs for every $\omega>0$, whereas emergence occurs only at the exceptional weight $\omega=1$. At this value, the multiplicity of the eigenvalue $1$ increases from $1$ to $2$: the direction spanned by $y$ is inherited, while the direction spanned by $\widetilde y$ is emergent. Moreover, every vector $\alpha y+\beta\widetilde y$ with $\beta\neq0$ is an emergent $1$-eigenvector of $L_{G+H}$. \end{example} 

Example~\ref{ex:motivating} shows how adding a new edge can produce emergence. The next example shows that emergence may also result from increasing the weight of an edge already present in the graph.

\begin{example}[Weight-dependent emergence
after increasing the weight of an existing edge] \label{ex:emergence-existing-edge} Consider the positively weighted graph $G$ with vertex set $\{1,2,3,4\}$, edge set \[ E = \bigl\{ \{1,2\},\{1,3\},\{1,4\},\{2,3\} \bigr\},\] 
and unit edge weights.
Let $L_G$ be its Laplacian. A direct calculation gives $\sigma(L_G)=\{0,1,3,4\}$, $E_4(L_G)=\operatorname{span}\{y\}$, where \[ y= \begin{bmatrix} -3\\1\\1\\1 \end{bmatrix}, \] and  \(E_3(L_G)=\operatorname{span}\{\widetilde y\}\), where \[ \widetilde y = \begin{bmatrix} 0\\-1\\1\\0 \end{bmatrix}. \] 

Let $H$ be the graph with the same vertex set as $G$, edge set $\{\{2,3\}\}$, and edge weight $\omega>0$. Note that the edge $\{2,3\}$ is already present in $G$. Let $b=e_2-e_3$, $\Omega=[\omega]$, and $L_{G+H}=L_G+\omega bb^T$. Since \(b^Ty=0\), we have
\[
L_{G+H}y
=
L_Gy+\omega b(b^Ty)
=
L_Gy
=
4y.
\]
Thus, \(y\) is an inherited \(4\)-eigenvector for every \(\omega>0\). On the other hand, \(b=-\widetilde y\), and hence
\(
bb^T\widetilde y=2\widetilde y.
\)
Therefore,
\[
L_{G+H}\widetilde y
=
L_G\widetilde y+\omega bb^T\widetilde y
=
(3+2\omega)\widetilde y.
\]

Because \(L_G\) is symmetric, eigenspaces corresponding to distinct
eigenvalues are orthogonal. Since \(b\) spans \(E_3(L_G)\), every vector
\(z\) in \(E_0(L_G)\), \(E_1(L_G)\), or \(E_4(L_G)\) satisfies
\(b^Tz=0\). Hence,
\[
L_{G+H}z
=
L_Gz+\omega b(b^Tz)
=
L_Gz.
\]
Thus, all eigenvectors of \(L_G\) corresponding to the eigenvalues
\(0\), \(1\), and \(4\) remain eigenvectors of \(L_{G+H}\) with the
same eigenvalues. The only original eigendirection affected by the
update is the direction spanned by \(\widetilde y\), whose eigenvalue
changes from \(3\) to \(3+2\omega\).
Consequently,
\[
E_4(L_{G+H})
=
\begin{cases}
\operatorname{span}\{y\},&\omega\neq\frac12,\\[1mm]
\operatorname{span}\{y,\widetilde y\},&\omega=\frac12.
\end{cases}
\]
Therefore, \(y\) is inherited for every \(\omega>0\), whereas emergence
occurs only when \(\omega=1/2\). At that value, \(\widetilde y\) is an
emergent \(4\)-eigenvector.

This shows that emergence can occur by increasing the weight of an  existing edge, without adding a new edge to the graph. 
The example also shows that an emergent $\lambda$-eigenvector need not be a completely new eigendirection of the original Laplacian: $\widetilde y$ is already an eigenvector of $L_G$, but for the different eigenvalue $3$. 
In general, however, an emergent eigenvector may have components in several eigenspaces of $L_G$.\end{example} 

The preceding examples show that inheritance and emergence can occur simultaneously. 
They also show that the existence of emergent eigenvectors may depend on the positive update weights, both when a
new edge is added and when the weight of an existing edge is increased. 
We now characterize inheritance and emergence in terms of the differences across the update edges.

For $y\in\mathbb R^n$, define its \emph{weighted edge-difference vector} by 
\begin{equation} \label{eq:weighted-edge-difference-vector}  X_\Omega^Ty = \Omega^{1/2}B^Ty. \end{equation}

 Its $\ell$th entry is $\sqrt{\omega_{i_\ell j_\ell}} (y_{i_\ell}-y_{j_\ell})$. Since every update weight is positive, \(X_{\Omega}^Ty=0\) if and only
if \(B^Ty=0\); equivalently, the entries of \(y\) agree at the endpoints
of every updated edge.

 \begin{theorem}[Inheritance--emergence dichotomy] \label{thm:inheritance-emergence-dichotomy} Let $\lambda\in\sigma(L_G)$ and suppose that  $0\neq y\in E_\lambda(L_{G+H})$. Then \(y\) is inherited if and only if \(X_\Omega^Ty=0\), or,
equivalently, \(B^Ty=0\). Hence, \(y\) is emergent if and only if
\(X_\Omega^Ty\neq0\), or, equivalently, \(B^Ty\neq0\).
 \end{theorem} 
 
 \begin{proof} Suppose first that $X_\Omega^T y=0$. Since $L_{G+H}=L_G+X_\Omega X_\Omega^T$ and $L_{G+H} y=\lambda y$, we obtain \[ (L_G-\lambda I_n)y = (L_{G+H}-\lambda I_n)y-X_\Omega (X_\Omega^Ty)  = 0. \] Thus, $y\in E_\lambda(L_G)$, and hence $y$ is inherited.
 
 Conversely, suppose that $y$ is inherited. Then $L_Gy=L_{G+H} y=\lambda y$, so $X_\Omega X_\Omega^Ty=(L_{G+H}-L_G)y=0$. Since $\ker(X_{\Omega}X_{\Omega}^T)= \ker (X_{\Omega}^T)$, we have $X_{\Omega}^Ty=0$.  Finally, since $X_\Omega^T=\Omega^{1/2}B^T$ and $\Omega^{1/2}$ is invertible, we have $X_\Omega^Ty=0$ if and only if $B^Ty=0$. This completes the proof of the characterization of inherited eigenvectors.
 
 The characterization of emergent eigenvectors is the negation of the inherited condition for a nonzero vector in $E_\lambda(L_{G+H})$. \end{proof} 

\section{The inherited eigenspace}
\label{sec:inherited-eigenspace}

Theorem~\ref{thm:inheritance-emergence-dichotomy} shows that $X_\Omega^Ty$ distinguishes inherited from emergent eigenvectors $y$.
This motivates the use of the following linear map:
\begin{equation}
\label{eq:inheritance-edge-difference-map}
T_{\lambda,\Omega}
:
E_\lambda(L_{G+H})\longrightarrow\mathbb R^p, \qquad T_{\lambda, \Omega}(y)=X_{\Omega}^Ty.
\end{equation}
We call $T_{\lambda,\Omega}$ the edge-difference map. 
By Theorem~\ref{thm:inheritance-emergence-dichotomy}, the nonzero vectors in its kernel are precisely the inherited $\lambda$-eigenvectors, that is,
\[
\mathcal U_\lambda=\ker(T_{\lambda,\Omega}).
\]

 Next, we give some equivalent characterizations of  $\mathcal U_\lambda$, which
 tell exactly which original $\lambda$-eigenvectors remain
$\lambda$-eigenvectors after the update.
We study this inheritance mechanism before turning to emergent eigenvectors.
Our result extends the eigenvector-preservation mechanism in Merris's Edge Principle~\cite{Merris1998} to multiple positively weighted edge updates.

\begin{proposition}[Inherited eigenvectors]\label{lem:inherited-eigenvectors}
Let \(\lambda\in\sigma(L_{G})\). For \(0\neq y\in E_\lambda(L_{G+H})\), the following
statements are equivalent:
\begin{enumerate}
\item[\rm(i)]  \(y\) is inherited.

\item[\rm(ii)] 
\(y\in\ker(B^T)\); that is, \(y_i=y_j\) for every update edge
\(\{i,j\}\in F\).

\item[\rm(iii)] The entries of \(y\) are constant on each connected
component of \(H\).
\end{enumerate}
\end{proposition}

\begin{proof}
Since $X_{\Omega}^Ty=0$ if and only if $B^Ty=0$, the equivalence of (i) and (ii) follows from Theorem~\ref{thm:inheritance-emergence-dichotomy}.

The entries of \(y\) agree at the endpoints of every update edge if and only if 
they agree along every path in the update graph (that is,  \(y\) is constant
on each connected component of \(H\)). 
This proves the equivalence of \emph{(ii)} and
\emph{(iii)}.
\end{proof}

We now use the preceding characterization to describe the entire inherited eigenspace and compute its dimension.

\begin{theorem}[The inherited eigenspace]
\label{thm:inherited-eigenspace}
  Let \(\lambda\in\sigma(L_G)\), let \(m=m_\lambda(L_G)\), and let
\(Q_\lambda\in\mathbb R^{n\times m}\) have orthonormal columns spanning
\(E_\lambda(L_G)\). Then the following equalities hold:
\begin{equation}
\label{eq:inherited-space-characterizations}
\mathcal U_\lambda
=
\ker(T_{\lambda,\Omega})
=
E_\lambda(L_G)\cap\ker(B^T)
=
\{Q_\lambda a:a\in\ker(B^TQ_\lambda)\}.
\end{equation}
In particular, the following are equivalent:
\begin{enumerate}
\item[(i)] $\mathcal{U}_{\lambda} \neq \{0\}$, that is, $\lambda$  is preserved through inheritance;
\item[(ii)] $B^TQ_{\lambda}$ does not have full column rank;
\item[(iii)] $Q_{\lambda}^TB\Omega B^T Q_{\lambda}$ is singular. 
\end{enumerate}
\end{theorem}

\begin{proof}
    The first and second characterizations of $\mathcal{U}_{\lambda}$ in \eqref{eq:inherited-space-characterizations}  follow from Theorem \ref{thm:inheritance-emergence-dichotomy} and Proposition  \ref{lem:inherited-eigenvectors}, respectively. The third equality in \eqref{eq:inherited-space-characterizations} follows since    every vector in \(E_\lambda(L_G)\) has the
form \(Q_\lambda a\) for a unique \(a\in\mathbb R^m\), and  $Q_{\lambda}$ has full column rank. 

The equivalence of (i) and (ii) follows from the last characterization of $\mathcal{U}_{\lambda}$ in \eqref{eq:inherited-space-characterizations} taking into account that  $Q_\lambda$ has full column rank.

Finally, let $A:=\Omega^{1/2}B^TQ_\lambda$.
Then
\[
Q_\lambda^TB\Omega B^TQ_\lambda
=
\bigl(\Omega^{1/2}B^TQ_{\lambda}\bigr)^T
\bigl(\Omega^{1/2}B^TQ_{\lambda}\bigr)
=
A^TA.
\]
Since $\Omega^{1/2}$ is invertible, $A$ has full column rank if and only if $B^TQ_\lambda$ has full column rank. Moreover, $A^TA$ is nonsingular if and only if $A$ has full column rank. This justifies the equivalence of
\textup{(ii)} and \textup{(iii)}.
\end{proof}

The dimension of the inherited eigenspace immediately follows from \eqref{eq:inherited-space-characterizations}.

\begin{corollary}[Dimension and preservation through inheritance]
\label{cor:inherited-dimension}
Let \(\lambda\in\sigma(L_G)\), let \(m=m_\lambda(L_G)\), and let
\(Q_\lambda\in\mathbb R^{n\times m}\) have orthonormal columns spanning
\(E_\lambda(L_G)\). Then
\begin{equation}
\label{eq:inherited-dimension}
\dim\mathcal U_\lambda
=
m_\lambda(L_G)-\rank(B^TQ_\lambda).
\end{equation}
In particular, if the number $p$ of updated edges is smaller than \(m_\lambda(L_G)\),  then \(\lambda\) is necessarily
preserved through inheritance.
\end{corollary}

\begin{proof}
Since $Q_{\lambda}$ has orthonormal columns, the map $a\mapsto Q_{\lambda}a$ is injective.
Therefore, by \eqref{eq:inherited-space-characterizations},
\(
\dim\mathcal U_\lambda=\dim\ker(B^TQ_{\lambda}).
\)
The rank--nullity theorem gives
\eqref{eq:inherited-dimension}.

Finally, if \(p<m_\lambda(L_G)\), then
\(
\rank(B^TQ_\lambda)\leq p<m_\lambda(L_G).
\)
It follows from \eqref{eq:inherited-dimension} that
\(\dim\mathcal U_\lambda>0\), so \(\lambda\) is preserved through
inheritance.
\end{proof}

Next, we make some interesting observations.

\begin{remark}[Independence from the positive update weights]
\label{rem:inheritance-weight-independence}
For a fixed set \(F\) of updated edges, the inherited eigenspace
$\mathcal U_\lambda$, and therefore its dimension, is independent of
the positive update weights in $\Omega$.
This contrasts with emergence, which may depend on \(\Omega\), as shown in
Examples~\ref{ex:motivating} and~\ref{ex:emergence-existing-edge}.
\end{remark}

\begin{remark}[Eigenvalue preservation under a single-edge update always involves inheritance]
\label{rem:single-edge-inheritance}
Suppose that the update consists of a single edge. Thus,
\(p=1\), \(B=b\in\mathbb R^n\), and
\(\Omega=[\omega]\) with \(\omega>0\), so that
\[
L_{G+H}=L_G+\omega bb^T.
\]
If
\(
\lambda\in\sigma(L_G)\cap\sigma(L_{G+H}),
\)
then necessarily
\(
\mathcal U_\lambda\neq\{0\},
\)
as we show next.
Hence, under a single positive edge update, every preserved eigenvalue
has a nonzero inherited eigenvector. This does not rule out the
simultaneous presence of emergent eigenvectors; it only shows that
a preserved eigenvalue cannot be preserved solely through emergence.

To see this, suppose, to the contrary, that
\(\mathcal U_\lambda=\{0\}\).
Choose any nonzero vectors
\(
x\in E_\lambda(L_G),
\) and \(
y\in E_\lambda(L_{G+H}).
\)
By Theorem~\ref{thm:inherited-eigenspace},
\(
\mathcal U_\lambda
=
E_\lambda(L_G)\cap\ker(b^T),
\)
so \(b^Tx\neq0\).
Since \(y\in E_\lambda(L_{G+H})\), we have
\begin{equation}\label{eq:eigenvalueeq}
(L_G-\lambda I_n)y
=
-\omega b(b^Ty).
\end{equation}
Because \(L_G\) is symmetric and \(x\in E_\lambda(L_G)\),
\(
x^T(L_G-\lambda I_n)
=
\bigl((L_G-\lambda I_n)x\bigr)^T
=
0.
\)
Multiplying \eqref{eq:eigenvalueeq} on the left by \(x^T\)
therefore gives
\[
0=-\omega(x^Tb)(b^Ty).
\]
Since \(\omega>0\) and \(x^Tb\neq0\), it follows that
\(b^Ty=0\).
Substituting this into \eqref{eq:eigenvalueeq} yields
\(
(L_G-\lambda I_n)y=0.
\)
Thus,
\[
0\neq y\in E_\lambda(L_G)\cap\ker(b^T)
=\mathcal U_\lambda,
\]
a contradiction.
Therefore, \(\mathcal U_\lambda\neq\{0\}\).
In particular, preservation solely through emergence requires at least
two updated edges.
\end{remark}

\begin{remark}[Geometric interpretation]
\label{rem:inheritance-geometric-interpretation}
Let \(b_1,\ldots,b_p\) be the columns of \(B\), and recall that
\(P_\lambda=Q_\lambda Q_\lambda^T\) is the orthogonal projector onto
\(E_\lambda(L_G)\). The vectors \(Q_\lambda^Tb_\ell\) are the coordinates of
\(P_\lambda b_\ell\) with respect to the orthonormal basis formed by the
columns of \(Q_\lambda\). Hence,
\begin{equation}
\label{eq:projected-edge-rank}
\rank(B^TQ_\lambda)=\rank(Q_{\lambda}^TB)=
\dim\left(
\operatorname{span}
\{Q_{\lambda}^Tb_1,\ldots,Q_{\lambda}^Tb_p\}
\right)
=
\dim\operatorname{span}
\{P_\lambda b_1,\ldots,P_\lambda b_p\}.
\end{equation}
If \(y\in E_\lambda(L_G)\), then \(P_\lambda y=y\). Since
\(P_\lambda\) is symmetric,
\(
b_\ell^Ty
=
b_\ell^TP_\lambda y
=
\left(P_\lambda b_\ell\right)^Ty
\). Therefore,
\begin{equation}
\label{eq:inherited-geometric-description}
\mathcal U_\lambda=E_{\lambda}(L_G)\cap \ker(B^T)
=
E_\lambda(L_G)\cap
\operatorname{span}
\{P_\lambda b_1,\ldots,P_\lambda b_p\}^{\perp}.
\end{equation}
Thus, \(\rank(B^TQ_\lambda)\) is the codimension of the inherited subspace
inside \(E_\lambda(L_G)\). In other words, it is the number of independent directions removed from the original $\lambda$-eigenspace by the update.
In particular, \(\lambda\) is preserved through inheritance if and only if
the projected update-edge vectors do not span \(E_\lambda(L_G)\).
\end{remark}

\medskip

\medskip

The following result gives a sufficient condition under which a preserved
eigenvalue has no emergent directions.

\begin{theorem}
\label{th1}
Let $L_G$ and $L_{G+H}=L_G+X_{\Omega}X_{\Omega}^{T}$ be the Laplacian matrices of $G$ and of an
updated graph of $G$, respectively. 
Let
\(\lambda\in\sigma(L_G)\cap\sigma(L_{G+H})\). Suppose that
\(
E_\mu(L_G)\subseteq E_\mu(L_{G+H})
\)
for every \(\mu\in\sigma(L_G)\) with \(\mu<\lambda\). Then 
\(
E_\lambda(L_{G+H})=\mathcal U_\lambda.
\)
\end{theorem}

\begin{proof}
Recall that $0$ is the smallest eigenvalue of $L_G$ and $L_{G+H}$. 
If $\lambda=0$,  we set $S=S'=\emptyset$. We use the convention that any vector is orthogonal to the empty set. 
For $\lambda >0$, define
\[
S:=\bigoplus_{\substack{\mu\in\sigma(L_{G})\\\mu<\lambda}}E_{\mu}(L_{G})\qquad\text{and}\qquad S^{\prime}:=\bigoplus_{\substack{\nu\in\sigma
(L_{G+H})\\\nu<\lambda}}E_{\nu}(L_{G+H}).
\]
Note that  $S\subseteq S^{\prime}$. Let $y$ be a unit vector in
$E_{\lambda}(L_{G+H})$ and let $C=X_{\Omega}X_{\Omega}^{T}$. Since $L_{G+H}$ is symmetric, we have $y\perp S^{\prime}$,
and therefore $y\perp S$. Then,\[
\lambda=\min_{x\bot S\text{,}\left\Vert x\right\Vert _{2}=1}(x^{T}L_Gx)\leq
y^{T}L_{G}y.
\]
Therefore,
\[
y^{T}L_{G}y+y^{T}Cy=y^{T}L_{G+H}y=\lambda\leq y^{T}L_{G}y.
\]
Since $C$ is positive semidefinite, we obtain $y^{T}Cy=0$. Since
$y^{T}Cy=\Vert X_{\Omega}^{T}y\Vert_{2}^{2}$, we have $X_{\Omega}^{T}y=0$.
Consequently,
\[
L_{G}y=L_{G+H}y-X_{\Omega}X_{\Omega}^{T}y=\lambda y.
\]
Thus $y\in E_{\lambda}(L_{G})$.
\end{proof}

\begin{corollary}
Let $L_G$ and $L_{G+H}=L_G+X_{\Omega}X_{\Omega}^{T}$ be the Laplacian matrices of $G$ and of an
updated graph of $G$, respectively. If $\lambda$ is one of the two smallest
eigenvalues (counting multiplicities) of $L_G$ and $\lambda\in\sigma(L_{G+H})$, then any $\lambda
$-eigenvector of $L_{G+H}$ is inherited.
\end{corollary}

\begin{proof}
If \(\lambda=0\), there are no eigenvalues of \(L_G\) smaller than
\(\lambda\). Thus, the hypothesis of Theorem~\ref{th1} is vacuously
satisfied, and the result follows.

Suppose now that \(\lambda\) is the second smallest eigenvalue of \(L_G\)
and that \(\lambda>0\). Then \(G\) is connected, so
\(
E_0(L_G)=\operatorname{span}\{\mathbf{1}\}.
\)
Since adding edges preserves connectivity, \(G+H\) is also connected, and
therefore
\(
E_0(L_{G+H})=\operatorname{span}\{\mathbf{1}\}.
\)
Thus, \(E_0(L_G)\subseteq E_0(L_{G+H})\). Since \(0\) is the only
eigenvalue of \(L_G\) smaller than \(\lambda\), the result follows from
Theorem~\ref{th1}.
\end{proof}

Having characterized the inherited eigenspace, we now study when the update produces new $\lambda$-eigendirections outside $E_\lambda(L_G)$.

\section{The orthogonal inherited--emergent decomposition} \label{sec:inherited-emergent-decomposition} 

In Theorem \ref{thm:inherited-eigenspace} we showed that the inherited subspace $\mathcal{U}_{\lambda}$ equals $\ker(T_{\lambda, \Omega})$, where $T_{\lambda, \Omega}$ is the edge-difference map    defined in \eqref{eq:inheritance-edge-difference-map}. On the other hand, by Theorem~\ref{thm:inheritance-emergence-dichotomy}, the nonzero vectors in $\operatorname{range}(T_{\lambda,\Omega})$ are precisely the edge-difference vectors of the emergent $\lambda$-eigenvectors. 
This motivates studying the range of $T_{\lambda,\Omega}$ to understand the emergent part of the updated eigenspace. 

The emergent \(\lambda\)-eigenvectors, together with the zero vector, do not in general form a subspace. To capture the directions in \(E_\lambda(L_{G+H})\) that are not accounted for by inheritance, we therefore consider the orthogonal complement of \(\mathcal U_\lambda\) within \(E_\lambda(L_{G+H})\), denoted by
\begin{equation} \label{def:ClambdaOmega} \mathcal C_{\lambda,\Omega} := E_\lambda(L_{G+H})\cap\mathcal U_\lambda^\perp.\end{equation} 
We will describe this subspace in terms of
\begin{equation} \label{def:SlambdaOmega} \mathcal S_{\lambda,\Omega} := \operatorname{range}(T_{\lambda,\Omega}) = \left\{ X_\Omega^Ty:y\in E_\lambda(L_{G+H}) \right\}. \end{equation} We call  \(\mathcal C_{\lambda,\Omega}\) the \emph{canonical emergent subspace} and \(\mathcal S_{\lambda,\Omega}\) the \emph{edge-difference space}. The edge-difference space records how updated \(\lambda\)-eigenvectors vary across the update edges. 

\medskip
A nonzero vector in $\mathcal S_{\lambda,\Omega}$ comes from an emergent $\lambda$-eigenvector, but it need not determine that eigenvector uniquely.
Indeed, $T_{\lambda,\Omega}$ cannot distinguish eigenvectors that differ by
an inherited vector. 
The next lemma makes this precise.

\begin{lemma}\label{lem:fibers-orthogonal-representatives} Let \(y_1,y_2\in E_\lambda(L_{G+H})\). Then \begin{equation} \label{eq:same-edge-difference-fiber} T_{\lambda,\Omega}(y_1)=T_{\lambda,\Omega}(y_2) \quad\text{if and only if}\quad y_1-y_2\in\mathcal U_\lambda. \end{equation} 
Consequently, if \(y\in E_\lambda(L_{G+H})\) and \(w=T_{\lambda,\Omega}(y)\), then
\begin{equation} 
\label{eq:same-edge-difference-set} 
\left\{z\in E_\lambda(L_{G+H}): T_{\lambda,\Omega}(z)=w\right\} = y+\mathcal U_\lambda. 
\end{equation}
Moreover, this set contains a unique vector in $\mathcal{C}_{\lambda, \Omega}.$

\end{lemma}

\begin{proof}
The equivalence in \eqref{eq:same-edge-difference-fiber} follows immediately from $\ker(T_{\lambda,\Omega})=\mathcal U_\lambda$. Then, 
\eqref{eq:same-edge-difference-set} immediately follows from \eqref{eq:same-edge-difference-fiber}.

It remains to prove the last statement.
Let $z\in(y+\mathcal{U}_{\lambda
}).$ Write $z=z_{1}+z_{2}$ with $z_{1}\in\mathcal{U}_{\lambda}$ and $z_{2}\in\mathcal{C}_{\lambda,\Omega}.$ Then, $T_{\lambda,\Omega}(z_{2})=T_{\lambda,\Omega}(z)=w.$ So, $z_{2}\in(y+\mathcal{U}_{\lambda})\cap\mathcal{C}_{\lambda,\Omega}$.

To prove uniqueness, suppose that $z^{\prime}\in(y+\mathcal{U}_{\lambda})\cap\mathcal{C}_{\lambda,\Omega}$. Then $T_{\lambda,\Omega}(z^{\prime
})=T_{\lambda,\Omega}(z_{2}),$ implying that $z^{\prime}-z_{2}\in
\mathcal{U}_{\lambda}$ by \eqref{eq:same-edge-difference-fiber}. Since $z^{\prime}-z_{2}\in\mathcal{C}_{\lambda,\Omega
}\subseteq \mathcal{U}_{\lambda}^{\perp},$ it follows that $z^{\prime}=z_{2}.$

\end{proof}

By Lemma~\ref{lem:fibers-orthogonal-representatives}, the restriction
\(
T_{\lambda,\Omega}\big|_{\mathcal C_{\lambda,\Omega}}
:
\mathcal C_{\lambda,\Omega}
\longrightarrow
\mathcal S_{\lambda,\Omega}
\)
is a linear isomorphism. We denote its inverse by
\begin{equation}
\label{def:Phi-lambda-Omega}
\Phi_{\lambda,\Omega}
:=
\left(
T_{\lambda,\Omega}\big|_{\mathcal C_{\lambda,\Omega}}
\right)^{-1}
:
\mathcal S_{\lambda,\Omega}
\longrightarrow
\mathcal C_{\lambda,\Omega}.
\end{equation}

Thus, for every $w\in\mathcal S_{\lambda,\Omega}$, $\Phi_{\lambda,\Omega}(w)$ is the unique vector  $y\in \mathcal C_{\lambda,\Omega}$ whose edge-difference vector $X_{\Omega}^Ty$ is $w$.

\medskip
Since \(\mathcal U_\lambda\) is a subspace of \(E_\lambda(L_{G+H})\), its orthogonal complement within this eigenspace is \(\mathcal C_{\lambda,\Omega}\). Together with \(\mathcal C_{\lambda,\Omega}=\Phi_{\lambda,\Omega}(\mathcal S_{\lambda,\Omega})\), this gives the following orthogonal decomposition.

\begin{theorem}[Orthogonal inherited--emergent decomposition]
\label{thm:updated-eigenspace-decomposition}
Let \(\lambda\in\sigma(L_G)\). Then
\begin{equation}
\label{eq:inherited-emergent-decomposition}
E_\lambda(L_{G+H})
=
\mathcal U_\lambda
\mathbin{\oplus^\perp}
\Phi_{\lambda,\Omega}
\bigl(\mathcal S_{\lambda,\Omega}\bigr).
\end{equation}
\end{theorem}

The orthogonal decomposition itself is a standard fact about subspaces of a finite-dimensional inner-product space. The useful point here is that the map $\Phi_{\lambda, \Omega}$  parametrizes $\mathcal{C}_{\lambda, \Omega}$ by the edge-difference vectors in
\(\mathcal S_{\lambda,\Omega}\). 
More precisely, every
\(y\in E_\lambda(L_{G+H})\) has a unique representation
\[
y
=
u+\Phi_{\lambda,\Omega}(w),
\qquad
u\in\mathcal U_\lambda,
\qquad
w\in\mathcal S_{\lambda,\Omega}.
\]
Moreover, the parameter $w$ is simply the edge-difference vector \(
w=T_{\lambda,\Omega}(y)
\) of $y$.

\medskip
Taking dimensions in \eqref{eq:inherited-emergent-decomposition} gives a formula for the multiplicity of \(\lambda\) in the updated Laplacian in terms of the dimension of $\mathcal{S}_{\lambda, \Omega} $. At this stage, however, the formula is structural rather than directly computable from the original eigenspace and the update data, as \(\mathcal S_{\lambda,\Omega}\) is  defined in terms of the eigenspace \(E_\lambda(L_{G+H})\).

\begin{corollary}[A preliminary multiplicity identity] \label{cor:preliminary-updated-multiplicity} Let \(\lambda\in\sigma(L_G)\) and let \(Q_\lambda\) have orthonormal columns spanning \(E_\lambda(L_G)\). Then \begin{equation} \label{eq:preliminary-updated-multiplicity} m_\lambda(L_{G+H}) = \dim\mathcal U_\lambda + \dim\mathcal S_{\lambda,\Omega} = m_\lambda(L_G) - \rank(B^TQ_\lambda) + \dim\mathcal S_{\lambda,\Omega}. \end{equation} 
\end{corollary} 

\begin{proof} Taking dimensions in \eqref{eq:inherited-emergent-decomposition} gives \[ m_\lambda(L_{G+H}) = \dim\mathcal U_\lambda + \dim \Phi_{\lambda,\Omega}(\mathcal S_{\lambda,\Omega}). \] The function  \(\Phi_{\lambda,\Omega}\)  is an isomorphism  with domain $\mathcal{S}_{\lambda, \Omega}$. Hence, \(\dim  \Phi_{\lambda,\Omega}(\mathcal S_{\lambda, \Omega})=\dim\mathcal S_{\lambda,\Omega}\). The final equality follows from Corollary~\ref{cor:inherited-dimension}. \end{proof} 

\section{A new characterization of $\mathcal{S}_{\lambda,\Omega}$ and reconstruction of  $\mathcal{C}_{\lambda,\Omega}$} \label{sec:reduced-characterization-reconstruction}

Section~\ref{sec:inherited-emergent-decomposition} gave the orthogonal decomposition of the updated eigenspace into its inherited and canonical emergent subspaces. 
It was also shown that
$\mathcal S_{\lambda,\Omega}$ parametrizes the canonical emergent subspace. 
However, $\mathcal S_{\lambda,\Omega}$ is defined in terms of the unknown updated eigenspace.
We now characterize this space directly from the original Laplacian and the update data. 
This makes the parameterization explicit and computable.

Rather than working directly with $\mathcal S_{\lambda,\Omega}$, we consider the rescaled edge-difference space
\begin{equation}
\label{def:ZlambdaOmega}
\mathcal Z_{\lambda,\Omega}
:=
\left\{
\Omega B^Ty:y\in E_\lambda(L_{G+H})
\right\}
=
\Omega^{1/2}\mathcal S_{\lambda,\Omega}.
\end{equation}

This rescaling will be particularly useful in the next section when we study the dependence of the structure of the updated eigenspace  on the positive update weights.

We start by characterizing $\mathcal Z_{\lambda,\Omega}$ using the nullspace of a symmetric matrix whose size is at most the number of updated edges,  which makes its calculation easier. 
This characterization provides a computable formula for the multiplicity of $\lambda$. 
It also gives a way to reconstruct the canonical emergent subspace and to obtain a basis of the updated eigenspace. 

\medskip
  We first introduce the space
\begin{equation}
\label{def:Klambda}
\mathcal K_\lambda
:=
\ker(Q_\lambda^TB).
\end{equation}
As we will show, the rescaled edge-difference space
$\mathcal Z_{\lambda,\Omega}$ is contained in $\mathcal K_\lambda$.
Thus, $\mathcal K_\lambda$ contains all rescaled edge-difference vectors that can produce emergent $\lambda$-eigenvectors.

Let
\[
d
:=
\dim\mathcal K_\lambda
=
p-\rank(B^TQ_\lambda),
\]
and fix a matrix $Z\in\mathbb R^{p\times d}$ whose columns form an orthonormal basis of $\mathcal K_\lambda$. 

We next introduce three matrices that will help us determine which vectors in $\mathcal K_\lambda$ belong to $\mathcal Z_{\lambda,\Omega}$. 
Define
\begin{align}
G_\lambda
&:=
(L_G-\lambda I_n)^\dagger,
\label{def:Glambda}
\\
\Gamma_\lambda
&:=
B^TG_\lambda B,
\label{def:gamma}
\\
F_\lambda(\Omega)
&:=
Z^T\bigl(\Omega^{-1}+\Gamma_\lambda\bigr)Z
\in\mathbb R^{d\times d}.
\label{eq:weight-matrix-pencil}
\end{align}
We will show that
\[
\mathcal Z_{\lambda,\Omega}
=
\{Zc:c\in\ker( F_\lambda(\Omega))\}.
\]
Thus, the nullspace of $F_\lambda(\Omega)$ parametrizes the rescaled edge-difference space $\mathcal{Z}_{\lambda,\Omega}$ and hence the canonical emergent subspace $\mathcal{C}_{\lambda, \Omega}$.

When $d=0$, we use the usual empty-matrix convention.
In particular, $\dim\ker F_\lambda(\Omega)=0$.

Note that $L_G-\lambda I_n$ is symmetric and has kernel $E_\lambda(L_G)$. 
Thus, its pseudoinverse $G_\lambda$ acts as its inverse on $E_\lambda(L_G)^\perp$ and is zero on $E_\lambda(L_G)$. 
It follows that
\begin{equation}
\label{eq:GlambdaxL-lambda}
(L_G-\lambda I_n)G_\lambda
=
G_\lambda(L_G-\lambda I_n)
=
P_{\lambda,\perp}.
\end{equation}

\begin{remark}
The matrix $F_\lambda(\Omega)$ is closely related to the classical
secular matrix for low-rank eigenvalue updates. If
$t\notin\sigma(L_G)$, the matrix determinant lemma gives
\begin{equation}
\label{eq:secular-determinant-identity}
\det(L_{G+H}-tI_n)
=
\det(L_G-tI_n)\det(\Omega)
\det\left(
\Omega^{-1}+B^T(L_G-tI_n)^{-1}B
\right).
\end{equation}
The $p\times p$ matrix
\[
\Omega^{-1}+B^T(L_G-tI_n)^{-1}B
\]
is the secular matrix. Since $Q_\lambda^TBZ=0$, its singular part at
$t=\lambda$ vanishes after multiplication by $Z^T$ and $Z$. Therefore,
\begin{equation}
\label{eq:F-as-compressed-secular-limit}
F_\lambda(\Omega)
=
\lim_{t\to\lambda}
Z^T\left(
\Omega^{-1}+B^T(L_G-tI_n)^{-1}B
\right)Z.
\end{equation}
Thus, $F_\lambda(\Omega)$ is the limit of the secular matrix after
restricting it to $\mathcal K_\lambda$.
\end{remark}

\medskip
 The next theorem gives two different   characterizations of $\mathcal{Z}_{\lambda, \Omega}$. One of them is written in terms of 
the nullspace of $F_\lambda(\Omega)$. 
Since $F_\lambda(\Omega)$ has size $d\times d$, where $d\leq p$, it reduces the problem of computing $\mathcal{C}_{\lambda, \Omega}$ to computing the kernel of a matrix of size  at most the number $p$ of updated edges.

\begin{theorem}[Characterization of the rescaled edge-difference space] \label{prop:edge-difference-space} With the notation above, \begin{equation} \label{eq:rescaled-edge-space-characterization} \mathcal Z_{\lambda,\Omega} = \left\{ z\in\mathcal K_\lambda: \bigl(\Omega^{-1}+\Gamma_\lambda\bigr)z \in\operatorname{range}(B^TQ_\lambda) \right\} = \left\{ Zc:c\in\ker F_\lambda(\Omega) \right\}. \end{equation} \end{theorem} 

\begin{proof}
Let \(z\in\mathcal Z_{\lambda,\Omega}\). By definition \eqref{def:ZlambdaOmega}, there exists
\(y\in E_\lambda(L_{G+H})\) such that \(z=\Omega B^Ty\). The eigenvalue
equation $L_{G+H}y= \lambda y$ can therefore be written as
\(
(L_G-\lambda I_n)y+Bz=0.
\)
Multiplying by \(Q_\lambda^T\) on the left gives \(Q_\lambda^TBz=0\), so
\(z\in\mathcal K_\lambda\).

Write \(y=Q_\lambda a+r\), where $a=Q_{\lambda}^Ty$ and
\(r\in E_\lambda(L_G)^\perp\). Then
\((L_G-\lambda I_n)r=-Bz\). Since \(z\in\mathcal K_\lambda=\ker(Q_{\lambda}^TB)\), the vector
\(Bz\) belongs to \(E_\lambda(L_G)^\perp\), and hence
\(
r=-G_\lambda Bz.
\)
Thus,
\begin{equation}
\label{eq:reduced-reconstruction-relations}
y=Q_\lambda a-G_\lambda Bz.
\end{equation}
Multiplying \eqref{eq:reduced-reconstruction-relations} by \(B^T\) and
using \(B^Ty=\Omega^{-1}z\) and
\(\Gamma_\lambda=B^TG_\lambda B\), we obtain
\begin{equation}
\label{eq:compatibility-condition}
B^TQ_\lambda a
=
(\Omega^{-1}+\Gamma_\lambda)z.
\end{equation}
Therefore,
\((\Omega^{-1}+\Gamma_\lambda)z\in
\operatorname{range}(B^TQ_\lambda)\).

Conversely, suppose that \(z\in\mathcal K_\lambda\) and
\(
(\Omega^{-1}+\Gamma_\lambda)z
\in\operatorname{range}(B^TQ_\lambda).
\)
Choose \(a\in\mathbb R^m\) satisfying
\eqref{eq:compatibility-condition}, and define \(y\) by
\eqref{eq:reduced-reconstruction-relations}. Then
\(B^Ty=\Omega^{-1}z\), so \(\Omega B^Ty=z\). Moreover, since $Q_{\lambda}^TBz=0$,
\[
(L_{G+H}-\lambda I_n)y
=
-P_{\lambda,\perp}Bz+Bz
=
P_\lambda Bz
=
0.
\]
Hence,
\(y\in E_\lambda(L_{G+H})\), and therefore
\(z\in\mathcal Z_{\lambda,\Omega}\). This completes the proof of the first equality in
\eqref{eq:rescaled-edge-space-characterization}.

Finally,
\(\operatorname{range}(B^TQ_\lambda)=\mathcal K_\lambda^\perp\).
Also, every \(z\in\mathcal K_\lambda\) can be written uniquely as
\(z=Zc\) for some $c\in \mathbb{R}^d$. Moreover,
\[
(\Omega^{-1}+\Gamma_\lambda)Zc\in\mathcal K_\lambda^\perp
\quad\Longleftrightarrow\quad
Z^T(\Omega^{-1}+\Gamma_\lambda)Zc=0.
\]
By the definition of \(F_\lambda(\Omega)\), the last condition is
\(F_\lambda(\Omega)c=0\), which proves the second equality.
\end{proof}

Theorem~\ref{prop:edge-difference-space} implies that $\dim\mathcal Z_{\lambda,\Omega}=\dim\ker F_\lambda(\Omega)$, since $Z$ has orthonormal columns.
Taking into account that  $\mathcal Z_{\lambda,\Omega}=\Omega^{1/2}\mathcal S_{\lambda,\Omega}$, and $\Omega^{1/2}$ is invertible, we have
 $\dim\mathcal S_{\lambda,\Omega}=\dim\ker F_\lambda(\Omega)$.
Combining this identity with the preliminary multiplicity formula in Corollary~\ref{cor:preliminary-updated-multiplicity} gives the following computable formula.

\begin{corollary}[Multiplicity after a positive edge update]
\label{cor:weighted-algebraic-multiplicity}
With the notation above,
\begin{equation}
\label{eq:weighted-algebraic-multiplicity}
m_\lambda(L_{G+H})
=
m_\lambda(L_G)
-\rank(B^TQ_\lambda)
+\dim \ker F_\lambda(\Omega).
\end{equation}
In particular, emergence occurs if and only if
\(\ker F_\lambda(\Omega)\neq\{0\}\).
\end{corollary}

Every term on the right-hand side of \eqref{eq:weighted-algebraic-multiplicity}
can be computed from \(L_G\), \(Q_\lambda\), and the update data \(B\) and
\(\Omega\), without prior knowledge of \(E_\lambda(L_{G+H})\).

\medskip

We now show how to reconstruct the vectors in
\(\mathcal{C}_{\lambda,\Omega}\) from vectors in
\(\ker(F_\lambda(\Omega))\). 

Let $c\in\ker F_\lambda(\Omega)$ and set $z_c:=Zc$.
By Theorem~\ref{prop:edge-difference-space},
$z_c\in\mathcal Z_{\lambda,\Omega}$. 
In particular,
\[
(\Omega^{-1}+\Gamma_\lambda)z_c
\in\operatorname{range}(B^TQ_\lambda).
\]
Therefore, the equation
\begin{equation}
\label{eq:compatibility-eq}
B^TQ_\lambda a
=
(\Omega^{-1}+\Gamma_\lambda)z_c
\end{equation}
has at least one solution. 
Let $a_c$ be its unique minimum-norm solution:
\begin{equation}\label{eq:ac}
a_c
:=
(B^TQ_\lambda)^\dagger
(\Omega^{-1}+\Gamma_\lambda)z_c.
\end{equation}
We now define
\begin{equation}
\label{eq:explicit-reconstruction-vectors}
y_c
:=
Q_\lambda a_c-G_\lambda Bz_c.
\end{equation}
Using \eqref{eq:compatibility-eq}, we obtain
\begin{align}
\Omega B^Ty_c
&=
\Omega\left(B^TQ_\lambda a_c-\Gamma_\lambda z_c\right)
=
\Omega\left(
(\Omega^{-1}+\Gamma_\lambda)z_c-\Gamma_\lambda z_c
\right)
=
z_c.
\label{eq:yc-zc}
\end{align}
Thus, $y_c$ has the rescaled edge-difference vector $z_c$. 
The next theorem shows that these vectors $y_c$ are precisely the vectors in the canonical emergent subspace.

\begin{theorem}[Reconstruction of the canonical emergent subspace]
\label{thm:explicit-reconstruction}
With the notation above,
\begin{equation}
\label{eq:canonical-emergent-reconstruction}
\mathcal C_{\lambda,\Omega}
=
\left\{
y_c:c\in\ker F_\lambda(\Omega)
\right\}.
\end{equation}
\end{theorem}

\begin{proof}
($\supseteq$) Let \(c\in\ker F_\lambda(\Omega)\). We  show that
\(y_c\in\mathcal C_{\lambda,\Omega}\).
By \eqref{eq:yc-zc}, $\Omega B^Ty_c=z_c$.  By definition of $Z$, $z_c:=Zc\in \mathcal{K}_{\lambda}$, and hence \(Q_\lambda^TBz_c=0\).
Therefore,
\[
\begin{aligned}
(L_{G+H}-\lambda I_n)y_c
&=
(L_G-\lambda I_n)
\bigl(Q_\lambda a_c-G_\lambda Bz_c\bigr)
+Bz_c\\
&=
-P_{\lambda,\perp}Bz_c+Bz_c
=
P_\lambda Bz_c
=
0.
\end{aligned}
\]
Thus, \(y_c\in E_\lambda(L_{G+H})\).

It remains to show that $y_c\perp\mathcal U_\lambda$. Since $a_c$ is
the minimum-norm solution of \eqref{eq:compatibility-eq}, $a_c\in\ker(B^TQ_\lambda)^\perp$.
Since $\mathcal U_\lambda
=
\left\{
Q_\lambda h:h\in\ker(B^TQ_\lambda)
\right\}$ by Theorem \ref{thm:inherited-eigenspace}, it follows that $Q_\lambda a_c\perp\mathcal U_\lambda$. Moreover,
$G_\lambda Bz_c\in \operatorname{range}({G_{\lambda}})=E_\lambda(L_G)^\perp$, while
$\mathcal U_\lambda\subseteq E_\lambda(L_G)$. Hence,
$G_\lambda Bz_c\perp\mathcal U_\lambda$. Therefore,
\[
y_c
=
Q_\lambda a_c-G_\lambda Bz_c
\perp\mathcal U_\lambda.
\]
Thus, $y_c\in\mathcal C_{\lambda,\Omega}$.

($\subseteq$) Let \(y\in\mathcal C_{\lambda,\Omega}\), and set
\(z:=\Omega B^Ty\). Then, by \eqref{def:ZlambdaOmega},
\(z\in\mathcal Z_{\lambda,\Omega}\), so
Theorem~\ref{prop:edge-difference-space} gives
\(z=Zc=z_c\) for a unique \(c\in\ker F_\lambda(\Omega)\). By \eqref{eq:yc-zc}, we have
\(\Omega B^Ty=\Omega B^Ty_c\). Because \(\Omega\) is invertible,
\(B^T(y-y_c)=0\). Moreover,
\(y-y_c\in E_\lambda(L_{G+H})\) because, by the first part of the proof, $y_c\in \mathcal{C}_{\lambda, \Omega}$. Hence,
\[
y-y_c\in
E_\lambda(L_{G+H})\cap\ker(B^T)
=
\mathcal U_\lambda.
\]
On the other hand,
\(y-y_c\in\mathcal U_\lambda^\perp\), because both \(y\) and \(y_c\)
belong to \(\mathcal C_{\lambda,\Omega}\). Therefore \(y-y_c=0\), and
hence \(y=y_c\).
This proves \eqref{eq:canonical-emergent-reconstruction}.
\end{proof}

The reconstruction theorem now gives the following procedure for computing a
basis of $E_{\lambda}(L_{G+H})$.

\begin{corollary}[A basis of the updated eigenspace]
\label{cor:basis-updated-eigenspace}
Let \(\{h_1,\ldots,h_s\}\) be a basis of
\(\ker(B^TQ_\lambda)\), and let \(\{c_1,\ldots,c_r\}\) be a basis of
\(\ker F_\lambda(\Omega)\). Then
\begin{equation}\label{eq:basis}
\{Q_\lambda h_1,\ldots,Q_\lambda h_s\}\cup\{
y_{c_1},\ldots,y_{c_r}\}
\end{equation}
is a basis of \(E_\lambda(L_{G+H})\), in which the first subset is a
basis of $\mathcal{U}_{\lambda}$ and the second is  a
basis of $\mathcal{C}_{\lambda, \Omega}$.
\end{corollary}

\begin{proof}
By Theorem~\ref{thm:inherited-eigenspace},
\(
\mathcal U_\lambda
=
\{Q_\lambda h:h\in\ker(B^TQ_\lambda)\}.
\)
Since $Q_{\lambda}$ has full column rank, $\{Q_\lambda h_1,\ldots,Q_\lambda h_s\}$ is a basis of $\mathcal{U}_{\lambda}.$

To prove the second part, we first note that the map \(c\mapsto y_c\) is linear. It is also injective. Indeed,  if $y_c=0$, then \eqref{eq:yc-zc} gives $Zc =0$. Since $Z$ has full column rank, $c=0.$
Moreover,
Theorem~\ref{thm:explicit-reconstruction} shows that the range of this
map is $\mathcal C_{\lambda,\Omega}$. Therefore, $\{y_{c_1},\ldots,y_{c_r}\}$ is a basis of $\mathcal C_{\lambda,\Omega}$.

Now the result follows since
\(
E_\lambda(L_{G+H})
=
\mathcal U_\lambda
\mathbin{\oplus^\perp}
\mathcal C_{\lambda,\Omega}.
\)
\end{proof}

The following example illustrates the usefulness of the reduced matrix $F_{\lambda}(\Omega)$, the multiplicity formula, and the reconstruction of the canonical emergent subspace.

\begin{example}[The reduced matrix for a one-edge update] \label{ex:weighted-emergent-dimension} Let \(G\) be the path on three vertices, with Laplacian \[ L_G= \begin{pmatrix} 1&-1&0\\ -1&2&-1\\ 0&-1&1 \end{pmatrix}. \] The eigenvalues of \(L_G\) are \(0,1,3\), and the unit vector \(q=\frac1{\sqrt6}(1,-2,1)^T\) spans \(E_3(L_G)\). Add the missing edge \(\{1,3\}\) with weight \(\omega>0\). Thus, \(B=b=e_1-e_3\), \(\Omega=[\omega]\), and \(L_{G+H}=L_G+\omega bb^T\). Since \(q^Tb=0\), \(\mathcal K_3=\mathbb R\), and we may take \(Z=[1]\). The vector \(b\) is a \(1\)-eigenvector of \(L_G\), so
\((L_G-3I_3)b=-2b\). Since \(L_G\) is symmetric,
\(b\perp E_3(L_G)\). On \(E_3(L_G)^\perp\), the pseudoinverse
\(G_3=(L_G-3I_3)^\dagger\) agrees with the inverse of the restriction of
\(L_G-3I_3\). Hence, \(G_3b=-\frac12b\).

Since \(b^Tb=2\), it follows that
\(\Gamma_3=b^TG_3b=-1\). Therefore, the reduced matrix is given by \[ F_3(\Omega)=\left[\frac1{\omega}-1\right]\in \mathbb{R}^{1\times 1}. \] 
By Corollary~\ref{cor:weighted-algebraic-multiplicity}, an emergent
\(3\)-eigendirection occurs precisely when
\(\ker F_3(\Omega)\neq\{0\}\), or equivalently, when
\(F_3(\Omega)\) is singular. Since \(F_3(\Omega)\) is a \(1\times1\)
matrix, this happens precisely when
\(\omega=1\).
Since \(b^Tq=0\), Theorem~\ref{thm:inherited-eigenspace} shows that the original 3-eigenvector
\(q\) is inherited  for every \(\omega>0\).
Corollary~\ref{cor:weighted-algebraic-multiplicity} therefore gives \[ m_3(L_{G+H}) = \begin{cases} 1,&\omega\neq1,\\ 2,&\omega=1. \end{cases} \] 

We now use the reconstruction formula to obtain $\mathcal{C}_{3, \Omega}$ when $\omega=1$. In this case, we have \(\ker F_3(\Omega)=\mathbb R\). For 
\(c\in\mathbb R\), we have \(z_c=c\). Moreover, since
\(B^TQ_3=b^Tq=0\), the minimum-norm solution in
\eqref{eq:ac} is \(a_c=0\). Hence,
\[
y_c=-G_3bc=\frac{c}{2}b.
\]
Thus, every nonzero \(c\) produces the same canonical emergent direction,
and Theorem~\ref{thm:explicit-reconstruction} gives
\[
\mathcal C_{3,\Omega}
=
\operatorname{span}\{b\}.
\]

 Consequently, \begin{equation} \label{eq:P3-reduced-eigenspace} E_3(L_{G+H}) = \begin{cases} \operatorname{span}\{q\}, &\omega\neq1,\\[1mm] \operatorname{span}\{q\} \mathbin{\oplus^\perp} \operatorname{span}\{b\}, &\omega=1. \end{cases} \end{equation} 

This conclusion can also be checked directly: \(L_{G+H} b=(1+2\omega)b\). Thus, \(b\) becomes a \(3\)-eigenvector precisely when \(\omega=1\). At this exceptional weight, the added edge turns the path into a cycle, whose nonzero Laplacian eigenvalue \(3\) has multiplicity two. 
\end{example}

\section{Weights that produce emergence} \label{sec:exceptional-weights}

Throughout this section, we fix the original Laplacian \(L_G\), an eigenvalue
\(\lambda\in\sigma(L_G)\), and the set of update edges, represented by an
oriented incidence matrix \(B\).
Section~\ref{sec:inherited-eigenspace} showed that, for fixed update edges, the inherited eigenspace \(\mathcal U_\lambda\) does not depend on the positive update weights.
Emergence is different. 
The canonical emergent subspace \(\mathcal C_{\lambda,\Omega}\), and hence the updated eigenspace \(E_\lambda(L_{G+H})\), may depend on \(\Omega\).

In this section, we study this dependence through three questions: when is emergence possible, where does it occur in the positive weight space, and how many emergent directions can occur?

 Among our results, we identify the set of  weights that produce emergence and show that it has Lebesgue  measure zero. 
 We also  study the dimension of the emergent subspace at these exceptional weights.

We now introduce the objects used to make these statements precise.
Recall the orthogonal decomposition
$$
E_\lambda(L_{G+H})
=
\mathcal U_\lambda
\mathbin{\oplus^\perp}
\mathcal C_{\lambda,\Omega}.
$$
Since \(\mathcal U_\lambda\) is independent of \(\Omega\), all dependence on
the update weights occurs through \(\mathcal C_{\lambda,\Omega}\).
By~\eqref{eq:canonical-emergent-reconstruction},
$$
\dim\mathcal C_{\lambda,\Omega}
=
\dim\ker F_\lambda(\Omega).
$$

To make the dependence of $F_\lambda(\Omega)$ on the weights explicit, we
write
\begin{equation}
\label{eq:F-weight-decomposition}
F_\lambda(\Omega)
=
Z^T\Omega^{-1}Z+M_\lambda,
\end{equation}
where
\begin{equation}
\label{def:Hlambda}
M_\lambda:=Z^T\Gamma_\lambda Z.
\end{equation}
The matrix $M_\lambda\in\mathbb R^{d\times d}$ is symmetric and depends on $L_G$, $\lambda$, and the update edges, but not on their weights.

To refer concisely to the weights that produce emergence, we introduce the following terminology. 
The reason for the term ``exceptional'' will become clear later, when we show that such weights form a set of measure zero in the positive weight space.

\begin{definition}
\label{def:exceptional-weight}
A positive diagonal matrix \(\Omega\) is called \emph{exceptional for
\(\lambda\)} if
$
\mathcal C_{\lambda,\Omega}\neq\{0\}.
$
The diagonal of an exceptional \(\Omega\) is called a \emph{vector of  exceptional weights}, or \emph{exceptional weights} when the context is clear.
\end{definition}

Thus, \(\Omega\) is exceptional if and only if $\ker F_\lambda(\Omega)\neq\{0\}$.

Throughout the section, we assume that \(d:=\dim\mathcal K_\lambda=\dim\ker(Q_\lambda^TB)>0\).
If \(d=0\), then \(F_\lambda(\Omega)\) is empty, so no emergence is possible by Theorem \ref{thm:explicit-reconstruction}, and hence no exceptional \(\Omega\) exists.

\subsection{Existence of exceptional weights}
We begin by determining when some positive choice of weights can produce emergence. 
The next theorem gives a complete characterization in terms of $M_\lambda$. Moreover, if emergence is possible, then it is already possible with equal update weights.

\begin{theorem}[Existence of exceptional weights] \label{prop:existence-exceptional-weights}  Fix $L_G$, $\lambda \in \sigma(L_G)$ and $B$. Suppose that \(d>0\). Then  the following statements are equivalent: \begin{enumerate} \item[\rm(i)] There is a positive diagonal matrix $\Omega$ exceptional for $\lambda.$
\item[\rm(ii)] The symmetric matrix \(M_\lambda\) has a negative eigenvalue. \item[\rm(iii)] There exists \(\omega>0\) such that \(\omega I_p\) is exceptional for \(\lambda\). \end{enumerate} Moreover, for every \(\omega>0\), \begin{equation} \label{eq:equal-weight-emergent-dimension} \dim\mathcal C_{\lambda,\omega I_p} = \dim E_{-1/\omega}(M_\lambda). \end{equation} Consequently, $\omega I_p$ is exceptional if and only if $-1/\omega$ is a negative eigenvalue of $M_\lambda$. 
In particular, there are at most $d$ exceptional equal weights.
\end{theorem} 

\begin{proof} We first show that (i) implies (ii). Assume that \(\Omega\) is exceptional. Then, by Theorem \ref{thm:explicit-reconstruction}, there exists a nonzero vector \(c\) such that \( F_\lambda(\Omega)c=0\). By \eqref{eq:F-weight-decomposition}, \[ c^TM_\lambda c = -c^TZ^T\Omega^{-1}Zc <0. \] 
The inequality is strict because \(\Omega^{-1}\) is positive definite and \(Zc\neq0\). Hence, since $M_{\lambda}$ is symmetric,  it follows that this matrix has a negative eigenvalue. 

We next show that (ii) implies (iii).  Suppose that \(M_\lambda c=\rho c\) for some \(c\neq0\) and \(\rho<0\). Set \(\omega=-1/\rho\). Since \(Z^TZ=I_d\), we have \(F_\lambda(\omega I_p)=\omega^{-1}I_d+M_\lambda\), and therefore \(F_\lambda(\omega I_p)c=0\). Thus, \(\omega I_p\) is exceptional,  by Theorem \ref{thm:explicit-reconstruction}. 

The implication from (iii) to (i) is immediate. 

Finally, let \(\omega>0\). By Theorem~\ref{thm:explicit-reconstruction},
\(
\dim \mathcal C_{\lambda,\omega I_p}
=
\dim\ker F_\lambda(\omega I_p).
\)
Since
\(
F_\lambda(\omega I_p)
=
\omega^{-1}I_d+M_\lambda,
\)
we have
\[
\ker F_\lambda(\omega I_p)
=
\ker\bigl(M_\lambda+\omega^{-1}I_d\bigr)
=
E_{-1/\omega}(M_\lambda).
\]
Therefore,
\[
\dim\mathcal C_{\lambda,\omega I_p}
=
\dim E_{-1/\omega}(M_\lambda),
\]
which proves \eqref{eq:equal-weight-emergent-dimension}.

It follows that $\omega I_p$ is exceptional if and only if $E_{-1/\omega}(M_\lambda)\neq\{0\}$ or, equivalently, if and only if $-1/\omega$ is an eigenvalue of $M_\lambda$. 
Since $\omega>0$, this eigenvalue is negative. 
Thus, each exceptional equal weight $\omega=-1/\rho$ is determined by a negative eigenvalue $\rho$ of $M_\lambda$.
Since $M_\lambda$ is a $d\times d$ matrix, there are at most $d$ such exceptional  weights.
\end{proof}

Theorem~\ref{prop:existence-exceptional-weights} shows that the equal-weight family detects whether emergence is possible. 
In fact, the same is true for any fixed relative choice of positive update weights. 
More precisely, if $\Omega_0$ is any fixed positive diagonal matrix, there exists a positive diagonal matrix exceptional for $\lambda$ if and only if $t\Omega_0$ is exceptional for some $t>0$.
The next result also characterizes all such values of $t$.

\begin{proposition}[Fixed relative update weights]
\label{prop:fixed-relative-weights}
 Fix $L_G$, $\lambda \in \sigma(L_G)$ and $B$. Suppose that $d>0$, and let $\Omega_0$ be a fixed positive diagonal matrix. 
Then there exists a positive diagonal matrix $\Omega$ exceptional for $\lambda$ if and only if $t\Omega_0$ is exceptional for some $t>0$.
Moreover, $t\Omega_0$ is exceptional if and only if $-1/t$ is a negative eigenvalue of
$$
\widetilde M_{\lambda,\Omega_0}
=
T_0^{-1/2}M_\lambda T_0^{-1/2},
\qquad
T_0=Z^T\Omega_0^{-1}Z.
$$
In particular, there are at most $d\leq p$ exceptional positive scale factors $t$.
\end{proposition}

\begin{proof}
Since $Z$ has full column rank and $\Omega_0$ is positive diagonal, $T_0$ is positive definite. 
Hence,
$$
F_\lambda(t\Omega_0)
=
t^{-1}T_0+M_\lambda
=
T_0^{1/2}
\left(
t^{-1}I_d+\widetilde M_{\lambda,\Omega_0}
\right)
T_0^{1/2}.
$$
Taking into account that $T_0^{1/2}$ is invertible, by Corollary~\ref{cor:basis-updated-eigenspace},
$$
\dim\mathcal C_{\lambda,t\Omega_0}
=
\dim\ker F_\lambda(t\Omega_0)
=
\dim\ker
\left(
t^{-1}I_d+\widetilde M_{\lambda,\Omega_0}
\right)
=
\dim E_{-1/t}
\bigl(\widetilde M_{\lambda,\Omega_0}\bigr).
$$

Therefore, $t\Omega_0$ is exceptional for $\lambda$ if and only if $-1/t$ is a negative eigenvalue of $\widetilde M_{\lambda,\Omega_0}$.

Moreover, $\widetilde M_{\lambda,\Omega_0}$ is congruent to $M_\lambda$. 
By Sylvester's law of inertia, $\widetilde M_{\lambda,\Omega_0}$ has a negative eigenvalue if and only if $M_\lambda$ has a negative eigenvalue. 
By Theorem~\ref{prop:existence-exceptional-weights}, this holds if and only if some positive diagonal matrix $\Omega$ is exceptional for $\lambda$.

Finally, since $\widetilde M_{\lambda,\Omega_0}$ is a $d\times d$ matrix, there are at most $d\leq p$ exceptional positive scale factors $t$.
\end{proof}

Theorem~\ref{prop:existence-exceptional-weights} gives a spectral condition for exceptional weights to exist. 
The next result gives an additional restriction on their size.

\begin{corollary}
 Fix $L_G$, $\lambda \in \sigma(L_G)$ and $B$. Assume that $\Omega$  is exceptional for $\lambda$.  Then $\lambda_{\min}(M_{\lambda})<0$ and 
$\omega_{\max}\geq 
-\frac{1}{\lambda_{\min}(M_\lambda)}$, where \(\omega_{\max}\) is the largest diagonal entry of
\(\Omega\).

\end{corollary}

\begin{proof}
The first claim follows immediately from Theorem \ref{prop:existence-exceptional-weights}. Now we prove the second claim. Since 
\(\Omega^{-1}\succeq \omega_{\max}^{-1}I_p\) and 
\(Z^TZ=I_d\),
\[
Z^T\Omega^{-1}Z
\succeq
\frac{1}{\omega_{\max}}I_d.
\]
Thus, by \eqref{eq:F-weight-decomposition}, we have
\[
F_\lambda(\Omega)
\succeq
\left(
\frac{1}{\omega_{\max}}
+
\lambda_{\min}(M_\lambda)
\right)I_d.
\]
Since $\Omega$ is exceptional for $\lambda$, $F_\lambda(\Omega)$ is singular by Corollary \ref{cor:weighted-algebraic-multiplicity}. 
Therefore, 
$$
\frac{1}{\omega_{\max}}
+
\lambda_{\min}(M_\lambda)
\leq 0,
$$
which completes the proof.
\end{proof}

We conclude this subsection by specializing Theorem \ref{prop:existence-exceptional-weights} to a single updated edge. 
In this case, the condition for emergence takes a particularly simple form.

\begin{corollary}[Single-edge updates]
\label{cor:single-edge-exceptional-weight}
Let $\lambda \in \sigma(L_G)$. Suppose that \(p=1\), and let \(b\) be the oriented update-edge vector.
An exceptional edge weight $\omega$ for $\lambda$ exists if and only if  \[Q_\lambda^Tb=0\quad \text{and} \quad
b^TG_\lambda b<0.
\]
In this case, the exceptional weight is unique and is given by
\[
\omega
=
-\frac{1}{b^TG_\lambda b}.
\]
\end{corollary}

\begin{proof}
If \(Q_\lambda^Tb\neq0\), then
\(\mathcal K_\lambda=\ker Q_{\lambda}^Tb=\{0\}\), and hence \(d=\dim \mathcal K_\lambda=0\) and  no exceptional weights exist. 

Suppose now that \(Q_\lambda^Tb=0\). Then
\(\mathcal K_\lambda=\mathbb R\), and we may take \(Z=[1]\). Thus,
\[
F_\lambda(\Omega)
=
\omega^{-1}+b^TG_\lambda b.
\]
By Theorem \ref{thm:explicit-reconstruction},  $\omega$ is exceptional for $\lambda$ if and only if $F_{\lambda}(\Omega)=0.$ Since $\omega>0$, this equation has a solution if and only if $b^TG_\lambda b<0$.
In this case, the solution is unique and is given by $\omega
=
-\frac{1}{b^TG_\lambda b}.
$ 
\end{proof}

So far, we have studied when exceptional weights exist and which weights are exceptional along fixed rays.
We now describe the exceptional set in the full positive weight space.
\subsection{Geometry of the exceptional weight set}

We now describe the set of exceptional weight vectors within the space \((0,\infty)^p\) of all positive weight assignments. We will show that it is the zero set of a nonzero polynomial and, consequently, has Lebesgue measure zero. 
We will then study the local structure of this zero set and relate it to the dimension of the canonical emergent subspace.

 For simplicity, in this section we write $\omega_k:=\omega_{i_kj_k}$ for the weight of the update edge $\{i_k,j_k\}$.

Recall that $\Omega$ is exceptional for $\lambda$ if and only if $\det F_\lambda(\Omega)=0$.
Since
$$
F_\lambda(\Omega)
=
Z^T\Omega^{-1}Z+M_\lambda,
$$
it is natural to introduce the reciprocal weights
$$
s_i:=\frac{1}{\omega_i},
\qquad i=1,\ldots,p.
$$
Writing $s=(s_1,\ldots,s_p)$, we define
\begin{equation}\label{widehatF}
\widehat{F}_{\lambda}(s):=F_{\lambda}(\Omega)= \sum_{i=1}^p (Z^Te_i e_i^TZ) s_i +M_{\lambda}. 
\end{equation}
We then define
\begin{equation}\label{def:plambda}
p_\lambda(s):=\det\widehat F_\lambda(s).
\end{equation}
Thus, $\Omega$ is exceptional if and only if $p_\lambda(s)=0$.

The next lemma gives two properties of $p_\lambda$ that we will need.

\begin{lemma}
\label{lem:polynomial-exceptional-weights}
The polynomial \(p_\lambda(s)\) is  not identically zero and is affine in each variable.

\end{lemma}

\begin{proof}
Fix all variables
except \(s_i\). Then there exists a matrix $A$ independent of $s_i$ such that
\[\widehat F_\lambda(s)=A+s_i\zeta_i\zeta_i^T, \qquad \zeta_i = Z^Te_i.\]
The rank-one determinant identity gives 
\[
\det(A+s_i\zeta_i\zeta_i^T)
=
\det(A)
+
s_i\zeta_i^T\operatorname{adj}(A)\zeta_i.
\]
Thus, for every $i$, \(p_\lambda\) is affine in \(s_i\). 

To show that $p_\lambda$ is not identically zero, set $s_1=\cdots=s_p=t$. Since $Z^TZ=I_d$,
$$
\widehat F_\lambda(t,\ldots,t)
=
tI_d+M_\lambda.
$$
Therefore,
$$
p_\lambda(t,\ldots,t)
=
\det(tI_d+M_\lambda),
$$
which is a monic polynomial of degree $d$. 
Hence, $p_\lambda$ is not identically zero.
\end{proof}

We now use the polynomial $p_{\lambda}(s)$ to describe the exceptional set. 
The next theorem shows, in particular, that emergence is nongeneric with respect to the positive update weights.

\begin{theorem}[Polynomial characterization of the exceptional set]
\label{cor:one-weight-exceptional}
Suppose that \(d>0\). 
Then a positive weight vector $(\omega_1, \ldots, \omega_p)$  is
exceptional for \(\lambda\) if and only if
\(p_\lambda(s_1, \ldots, s_p)=0\), where $s_i=1/\omega_i$.
Consequently, the set of  exceptional  weight vectors has 
Lebesgue measure zero in \((0,\infty)^p\).
\end{theorem}

\begin{proof}
The first claim follows from the definition of \(p_\lambda\).
 By Lemma~\ref{lem:polynomial-exceptional-weights},
\(p_\lambda\) is a nonzero polynomial.
Clearing denominators in
\(p_\lambda(1/\omega_1,\ldots,1/\omega_p)\) and taking into account that $p_{\lambda}$ is affine in each variable, we get the nonzero
polynomial
\begin{equation}\label{eqhatqlambda}
\widehat{p}_\lambda(\omega_1,\ldots,\omega_p)
:=
(\omega_1\cdots\omega_p)
p_\lambda\left(\frac{1}{\omega_1},\ldots,\frac{1}{\omega_p}\right).
\end{equation}
Thus, a positive
weight vector is exceptional if and only if it is a zero
of \(\widehat{p}_\lambda\).
The conclusion follows because the zero set of a nonzero
real polynomial has Lebesgue measure zero.
\end{proof}
Thus, for fixed $L_G$, $\lambda$, and update edges, emergence occurs only for a measure-zero set of positive update weights.
This justifies the term \emph{exceptional weights}: for almost every choice of positive weights, the canonical emergent subspace is zero.

\begin{remark}[An alternative determinant formulation]
\label{rem:alternative-determinant}

The preceding theorem characterizes exceptional weights using
\(p_\lambda\), a polynomial in the reciprocal weights \(s_i\).
A polynomial characterization can also be obtained directly in the update
weights \(\omega_i\). A natural first attempt is to use
\[
\det\bigl(L_G-\lambda I_n+B\Omega B^T\bigr)
\]
for this purpose. If \(\mathcal U_\lambda\neq\{0\}\), however,
this determinant vanishes for every positive diagonal matrix
\(\Omega\), because every inherited vector belongs to the kernel of
\(L_G-\lambda I_n+B\Omega B^T\).
Thus, the full determinant does not distinguish emergence from
preservation through inheritance.

To remove the inherited subspace, let
\(h:=\dim\mathcal U_\lambda\), and choose
\(V\in\mathbb R^{n\times(n-h)}\) with orthonormal columns spanning
\(\mathcal U_\lambda^\perp\). Since \(\mathcal U_\lambda\) is
independent of the positive update weights, \(V\) can also be
chosen independently of them. Define
\[
q_\lambda(\omega_1,\ldots,\omega_p)
:=
\det\left(
V^T\bigl(L_G-\lambda I_n+B\Omega B^T\bigr)V
\right).
\]
When \(\mathcal U_\lambda=\{0\}\), we may take \(V=I_n\), so
\(q_\lambda\) reduces to the full determinant above. Since \(L_{G+H}-\lambda I_n\) is symmetric and vanishes
on \(\mathcal U_\lambda\), it leaves
\(\mathcal U_\lambda^\perp\) invariant. Consequently,
multiplication by \(V\) maps the kernel of
\(
V^T\bigl(L_G-\lambda I_n+B\Omega B^T\bigr)V
\)
isomorphically onto \(\mathcal C_{\lambda,\Omega}\). Hence,
\[
q_\lambda(\omega_1,\ldots,\omega_p)=0
\quad\Longleftrightarrow\quad
\Omega\text{ is exceptional for }\lambda.
\]

Moreover, \(q_\lambda\) and \(p_\lambda\) give the same polynomial description of the exceptional set, up to the reciprocal change of variables and a nonzero constant factor. More precisely,
\[
q_\lambda(\omega_1,\ldots,\omega_p)
=
c_\lambda(\omega_1\cdots\omega_p)
p_\lambda\left(
\frac{1}{\omega_1},\ldots,\frac{1}{\omega_p}
\right),
\]
where \(c_\lambda\neq0\) is independent of the update weights. 
This identity is stronger than the fact that the two polynomials have the same positive zero set and does not follow from that fact alone. 
It can be proved by a block determinant calculation using orthonormal bases adapted to \(\mathcal U_\lambda\) and \(\mathcal K_\lambda\), together with Schur complements.

Since \(p_\lambda\) is nonzero, the identity shows that \(q_\lambda\) is also a nonzero polynomial.
Its positive zeros are precisely the exceptional weight vectors.

We work with \(p_\lambda\) because it arises from the reduced
matrix \(F_\lambda(\Omega)\), whose order is \(d\leq p\), compared
with \(n-h\) for the matrix defining \(q_\lambda\). Moreover,
\(F_\lambda(\Omega)\) is the same reduced matrix used to determine
emergence and its multiplicity and to reconstruct the canonical
emergent subspace. Its explicit dependence on the reciprocal
weights therefore provides a convenient unified framework for
the analysis in this section.

\end{remark}

The fact that the polynomial $p_\lambda$ is affine in each variable also restricts the shape of the exceptional set along each coordinate direction. 
Once all but one weight are fixed, there is at most one exceptional value of the remaining weight, unless all positive values are exceptional.

\begin{corollary}\label{cor:fix-one-variable}
 Fix all update weights except one, say \(\omega_i\). Then the set of
positive values of \(\omega_i\) for which the resulting weight vector $(\omega_1, \ldots, \omega_p)$
is exceptional is empty, a singleton, or all of \((0,\infty)\).

\end{corollary}

\begin{proof}

Fixing all the weights except \(\omega_i\) is equivalent to fixing all
reciprocal variables except \(s_i=1/\omega_i\). 
By
Lemma~\ref{lem:polynomial-exceptional-weights},
\(p_\lambda\) is affine in \(s_i\), that is,
\(
p_\lambda(s)=as_i+b,
\)
where $a$ and $b$  are constant. Thus, the set of positive zeros of $p_{\lambda}(s)$  is empty, a singleton, or all of $(0,\infty)$. Since $s_i=1/\omega_i$ is a bijection from $(0,\infty)$ onto itself, the same holds for the exceptional values of $\omega_i$.
\end{proof}

The
following small example gives an explicit description of the set of vectors of exceptional weights.

\begin{example}\label{ex:equal-unequal-exceptional-weights} 
Consider the path on three vertices  with unit edge weights, and its Laplacian eigenvalue \(\lambda=3\). Increase the weights of the two existing edges \(\{1,2\}\) and \(\{2,3\}\) by \(\omega_1,\omega_2>0\), respectively. We can take
\[ B= \begin{bmatrix} 1&0\\ -1&1\\ 0&-1 \end{bmatrix}, \qquad \Omega=\operatorname{diag}(\omega_1,\omega_2). \] 
The vector \(q=\frac1{\sqrt6}(1,-2,1)^T\) spans \(E_3(L_G)\), and \(\mathcal K_3=\operatorname{span}\{(1,1)\}\). Taking \(Z=\frac1{\sqrt2}(1,1)^T\), one obtains 
\[ \Gamma_3 = -\frac14 \begin{bmatrix} 1&1\\ 1&1 \end{bmatrix}, \qquad M_3=\left[-\frac12\right], \qquad F_3(\Omega) = \left[\frac12 \left( \frac1{\omega_1}+\frac1{\omega_2}-1 \right)\right]. \] 
Therefore, emergence occurs precisely when \begin{equation} \label{eq:P3-exceptional-weight-curve} \frac1{\omega_1}+\frac1{\omega_2}=1, \end{equation} 
or equivalently, 
$$s_1+s_2=1.$$

This equation has infinitely many positive solutions.  At every point on the exceptional curve, \(F_3(\Omega)=[0]\). Since \(F_3(\Omega)\) is \(1\times1\), it follows that \(\dim\ker F_3(\Omega)=1\). Hence \(\dim\mathcal C_{3,\Omega}=1\), and every exceptional weight vector produces exactly one canonical emergent direction.
 The unique exceptional equal weight is \(\omega_1=\omega_2=2\), in agreement with Theorem~\ref{prop:existence-exceptional-weights}. Moreover, since  \(B^Tq\neq0\), by Theorem \ref{thm:inherited-eigenspace},  no nonzero \(3\)-eigenvector is inherited. \end{example} 

Theorem~\ref{cor:one-weight-exceptional} characterizes the weights for which emergence occurs. 
We now show that the local properties of $p_\lambda$ at an exceptional point determine the dimension of the canonical emergent subspace.

We first recall some terminology for polynomial zero sets.  Let \(s_0\)
belong to the zero set of a polynomial \(p\). If
\(\nabla p(s_0)\neq0\), then the implicit function theorem shows that one coordinate can be
expressed locally as a smooth function of the other \(p-1\)
coordinates. In this case, \(s_0\) is called a \emph{regular point} of
the zero set. If $\nabla p(s_0)=0$, the point $s_0$ is called \emph{singular}.

For a nonempty subset
\(J=\{i_1,\ldots,i_k\}\subseteq\{1,\ldots,p\}\), define
\[
\partial_Jp_\lambda(s)
:=
\frac{\partial^k p_\lambda}
{\partial s_{i_1}\cdots\partial s_{i_k}}(s).
\]

\begin{theorem}[Multiplicity of emergence at exceptional weights]
\label{thm:higher-nullity-strata}
Let
\(
\Omega=\operatorname{diag}(\omega_1,\ldots,\omega_p)
\)
be any positive diagonal matrix such that
\(\mathcal C_{\lambda,\Omega}\neq\{0\}\), and set
\(
s_0=
\left(
\frac1{\omega_1},\ldots,\frac1{\omega_p}
\right).
\)

Then, for every \(r\in\{2,\ldots,d\}\),
\(
\dim\mathcal C_{\lambda,\Omega}\geq r
\)
if and only if
\(
\partial_Jp_\lambda(s_0)=0
\)
for every subset \(J\subseteq\{1,\ldots,p\}\) satisfying
\(
1\leq |J|\leq r-1.
\)

Moreover,
\(
\dim\mathcal C_{\lambda,\Omega}=1
\)
if and only if \(s_0\) is a regular point of the zero set of
\(p_\lambda\).
\end{theorem}

\begin{proof}
 By
Corollary~\ref{cor:basis-updated-eigenspace} and \eqref{widehatF},
\[
k:=\dim\ker \widehat F_\lambda(s_0)=\dim\mathcal C_{\lambda,\Omega}\geq1.
\]
In particular, \(p_\lambda(s_0)=\det \widehat F_\lambda(s_0)=0\).
For \(i=1,\ldots,p\), set \(\zeta_i:=Z^Te_i\).

Suppose that \(k\geq r\), and let
\(J=\{j_1,\ldots,j_m\}\subseteq\{1,\ldots,p\}\) satisfy
\(1\leq m=|J|\leq r-1\).
Define the polynomial in the variables \(t_j\), \(j\in J\), by
\[
g_J(t):=
p_\lambda\left(s_0+\sum_{j\in J}t_je_j\right).
\]
We first show that \(g_J\) is identically zero.
By \eqref{widehatF}, for any scalars \(a_j\), $j\in J$,
\[
\widehat F_\lambda\left(s_0+\sum_{j\in J}a_je_j\right)
=
\widehat F_\lambda(s_0)+\sum_{j\in J}a_j\zeta_j\zeta_j^T.
\]
The rank--nullity theorem gives \(\operatorname{rank}\widehat F_\lambda(s_0)=d-k\). Also,
 each matrix \(a_j\zeta_j\zeta_j^T\) has rank at most one.
Since \(k\geq r>|J|\), we obtain
\[
\begin{aligned}
\operatorname{rank}\left(
\widehat F_\lambda(s_0)+\sum_{j\in J}a_j\zeta_j\zeta_j^T
\right)
\leq \operatorname{rank}\widehat F_\lambda(s_0)+|J|
=d-k+|J|<d.
\end{aligned}
\]
Thus, the matrix $\widehat F_\lambda(s_0)+\sum_{j\in J}a_j\zeta_j\zeta_j^T$ is singular for every choice of the scalars
\(a_j\), and therefore,
\[
g_J(t)
=
\det\left(\widehat F_\lambda(s_0)+\sum_{j\in J}t_j\zeta_j\zeta_j^T\right)
\equiv
0.
\]
Each variable \(t_j\) changes only the \(j\)th coordinate of
the argument of \(p_\lambda\). Hence, by the chain rule,
\[
\frac{\partial^{m}g_{J}}{\partial t_{j_{1}}\cdots\partial t_{j_{m}}}(t)=\frac{\partial^{m}p_{\lambda}}{\partial s_{j_{1}}\cdots\partial
s_{j_{m}}}\left(  s_{0}+\sum_{j\in J}t_{j}e_{j}\right)  .
\]
Thus, 
\[
0=\frac{\partial^{m}g_{J}}{\partial t_{j_{1}}\cdots\partial t_{j_{m}}}(0)=\frac{\partial^{m}p_{\lambda}}{\partial s_{j_{1}}\cdots\partial
s_{j_{m}}}(s_{0})=\partial_{J}p_{\lambda}(s_{0}),
\]
where the first equality holds because \(g_J\) is identically
zero. This proves the ``only if'' implication.

Conversely, suppose that \(\partial_Jp_\lambda(s_0)=0\) for
every \(J\subseteq\{1,\ldots,p\}\) satisfying
\(1\leq|J|\leq r-1\).
Since the columns of \(Z\) are orthonormal,
\[
\sum_{i=1}^p\zeta_i\zeta_i^T
=
Z^T\left(\sum_{i=1}^pe_ie_i^T\right)Z
=
Z^TZ
=
I_d.
\]
Consequently, the polynomial, in the real variable $t$,
\(
q(t):=p_\lambda(s_0+t\mathbf1_p)
\)
satisfies
\[
q(t)
=
\det\left(\widehat F_\lambda(s_0)+t\sum_{i=1}^p\zeta_i\zeta_i^T\right)
=
\det(\widehat F_\lambda(s_0)+tI_d).
\]

Recall that \(p_\lambda\) is affine in each variable.
Thus, every partial derivative in which the same variable
is differentiated twice is zero.
Repeated application of the chain rule therefore gives
\begin{equation}\label{derivatives of q}
q^{(\ell)}(0)
=
\ell!
\sum_{\substack{J\subseteq\{1,\ldots,p\}\\|J|=\ell}}
\partial_Jp_\lambda(s_0),
\qquad 1\leq\ell\leq r-1.
\end{equation}
The factor \(\ell!\) accounts for the possible orders of
differentiation with respect to the \(\ell\) distinct variables
indexed by \(J\).
By assumption, every term in the sum is zero, so
\(q^{(\ell)}(0)=0\) for \(1\leq\ell\leq r-1\).
Since also \(q(0)=p_\lambda(s_0)=0\), we conclude that
\(q(t)\) is divisible by \(t^r\).

The matrix \(\widehat F_\lambda(s_0)\) is real symmetric and therefore has an
orthonormal basis of eigenvectors. Since \(\dim\ker \widehat F_\lambda(s_0)=k\),
exactly \(k\) of its eigenvalues are zero, counted with
multiplicity.
Let \(\mu_1,\ldots,\mu_{d-k}\) be its nonzero eigenvalues,
also counted with multiplicity. Then
\begin{equation}\label{eq:qfactorization}
q(t)
=
\det(\widehat F_\lambda(s_0)+tI_d)
=
t^k\prod_{j=1}^{d-k}(\mu_j+t),
\end{equation}
where an empty product is understood to equal \(1\).
Since each \(\mu_j\neq0\), the product is nonzero at \(t=0\).
Thus, \(q(t)\) is divisible by \(t^k\) but not by \(t^{k+1}\).
Since it is also divisible by \(t^r\), we conclude that
\(k\geq r\),   proving the ``if'' claim.

For \(d\geq2\), the final assertion follows directly from the
equivalence just proved with \(r=2\), since \(k\geq1\) and $p_{\lambda}(s_0)=0.$

It remains to consider \(d=1\). In this case, \(1\leq k\leq d\)
implies \(k=1\), so \(\widehat F_\lambda(s_0)\) is the
\(1\times1\) zero matrix. Since \(Z^TZ=I_1\), we have
\[
q(t)=p_\lambda(s_0+t\mathbf1_p)
=
\det\bigl(\widehat F_\lambda(s_0)+tI_1\bigr)
=
t.
\]
Differentiating with respect to \(t\), evaluating at \(t=0\), and using \eqref{derivatives of q}
gives
\[
\sum_{i=1}^p
\frac{\partial p_\lambda}{\partial s_i}(s_0)
=
1.
\]
Thus, at least one first partial derivative is nonzero,
and \(s_0\) is a regular point of the zero set of \(p_\lambda\).
Since \(\dim\mathcal C_{\lambda,\Omega}=k=1\), the final
assertion also holds when \(d=1\).
\end{proof}

The following example illustrates the preceding results on a graph with 1000 vertices but only three update edges. The 
exceptional weight vectors are described implicitly by the polynomial $p_{\lambda}(s_1, s_2, s_3)$, and we identify an exceptional weight vector for which $\mathcal C_{\lambda,\Omega}$ has dimension $2$. 
The analysis reduces to the $3\times3$ matrix $F_\lambda(\Omega)$.

\begin{example}[A large cycle with three updated edges]
\label{ex:large-cycle-three-edges}
Let \(G\)  be the cycle on 1000 vertices with unit edge weights. Consider its Laplacian eigenvalue
\[
\lambda
=
2-2\cos\left(\frac{8\pi}{1000}\right)
\approx 6.316214334\times10^{-4}.
\]
The corresponding eigenspace has dimension two and is spanned by the vectors $u$ and $v$ whose
\(j\)th entries are
\[
u_j=\sqrt{\frac{2}{1000}}
\cos\left(\frac{8\pi(j-1)}{1000}\right)
\quad\text{and}\quad
v_j=\sqrt{\frac{2}{1000}}
\sin\left(\frac{8\pi(j-1)}{1000}\right).
\]
Let $Q_{\lambda}=[u, v].$ We add the three diameter edges
\[
\{1,501\},\qquad
\{101,601\},\qquad
\{201,701\},
\]
and let \(B\in\mathbb R^{1000\times3}\) contain their oriented edge
vectors. The entries of both \(\lambda\)-eigenvectors agree at the
endpoints of each update edge. Hence, \(Q_\lambda^TB=0\). Thus, \(\mathcal K_\lambda=\mathbb R^3\), and we take \(Z=I_3\). By
Theorem~\ref{thm:inherited-eigenspace}, the original two-dimensional
\(\lambda\)-eigenspace is inherited for every positive choice of the
three update weights. 

Let $b_i$ denote the $i$th column of $B$, as before. A computation gives
\[
b_i^TG_\lambda b_j
=
\frac{4}{1000}
\sum_{\substack{m=0\\m\ {\rm odd}}}^{999}
\frac{
\cos\!\left(
2\pi m(r_i-r_j)/1000
\right)
}{
\lambda_m-\lambda
},
\]
where \(r_1=0\), \(r_2=100\), \(r_3=200\), and
\(\lambda_m=2-2\cos(2\pi m/1000)\).
Evaluating these sums,  we obtain
\[
M_\lambda
=
\Gamma_\lambda
=
\begin{pmatrix}
0&-\alpha&\beta\\
-\alpha&0&-\alpha\\
\beta&-\alpha&0
\end{pmatrix},
\]
where
\[
\alpha
=
\frac{\sin(\pi/5)}{\sin(\pi/125)}
\approx23.3896943853,
\qquad
\beta
=
\frac{\sin(2\pi/5)}{\sin(\pi/125)}
\approx37.8453205018.
\]
Consequently, for
\(\Omega=\operatorname{diag}(\omega_1,\omega_2,\omega_3)\),
\begin{equation}
\label{eq:large-cycle-reduced-matrix}
F_\lambda(\Omega)
=
\begin{pmatrix}
\omega_1^{-1}&-\alpha&\beta\\
-\alpha&\omega_2^{-1}&-\alpha\\
\beta&-\alpha&\omega_3^{-1}
\end{pmatrix}.
\end{equation}
In the reciprocal variables \(s_i=1/\omega_i\), its determinant is
\begin{equation}
\label{eq:large-cycle-exceptional-polynomial}
p_\lambda(s)
=
s_1s_2s_3
-\alpha^2(s_1+s_3)
-\beta^2s_2
+2\alpha^2\beta.
\end{equation}

This polynomial displays two features of the general theory. First, if we fix
\(s_1=s_3=\beta\), then
\(p_\lambda(s)=0\) for every \(s_2>0\) (recall Corollary \ref{cor:fix-one-variable}). By Theorem \ref{cor:one-weight-exceptional}, if
\[
\omega_1=\omega_3=\frac1\beta
\approx0.0264233461,
\]
then every positive value of \(\omega_2\) is exceptional. 
Now consider
\[
\begin{aligned}
s_*&=\left(\beta,\frac{\alpha^2}{\beta},\beta\right),\\
\Omega_*&=\operatorname{diag}
\left(\frac1\beta,\frac{\beta}{\alpha^2},\frac1\beta\right)\\
&\approx\operatorname{diag}
(0.0264233461,0.0691772181,0.0264233461).
\end{aligned}
\]
At this point,
\[
F_\lambda(\Omega_*)
=
\begin{pmatrix}
\beta&-\alpha&\beta\\
-\alpha&\alpha^2/\beta&-\alpha\\
\beta&-\alpha&\beta
\end{pmatrix}
=
qq^T,
\qquad
q=
\begin{pmatrix}
\sqrt{\beta}\\
-\alpha/\sqrt{\beta}\\
\sqrt{\beta}
\end{pmatrix}.
\]
Therefore, \(F_\lambda(\Omega_*)\) has rank one and nullity two. Hence, by Corollary \ref{cor:basis-updated-eigenspace},
\(\dim\mathcal C_{\lambda,\Omega_*}=2\). Since the original
two-dimensional \(\lambda\)-eigenspace is inherited, the multiplicity of
\(\lambda\) in the updated Laplacian is four. 

Table~\ref{tab:large-cycle-weights} reports the dimension of $\mathcal{C}_{\lambda, \Omega}$ for  several graph updates $G+H$, three of which correspond to exceptional weight vectors. In each case, the multiplicity of $\lambda$ as an eigenvalue of $L_{G+H}$ is obtained by adding $2= \dim(\mathcal{U}_{\lambda})$ to the dimension of $\mathcal{C}_{\lambda, \Omega}.$ 
For this purpose, we computed
the kernel of the \(3\times 3\) matrix \(F_\lambda(\Omega)\). We note that the term \(M_\lambda\), which involves the computation of the pseudoinverse of a $1000\times 1000$ matrix, needs to be computed only once, since it is independent of $\Omega.$ Then,  many choices of
the three update weights can be tested without repeatedly computing the relevant eigenspace of the
\(1000\times1000\) updated Laplacian.

The exceptional $\Omega$'s with equal diagonal entries appearing in Table~\ref{tab:large-cycle-weights} are obtained from the two negative eigenvalues
of \(M_\lambda\), according to Theorem \ref{prop:existence-exceptional-weights}. 

\begin{table}
\tbl{Dimensions of $C_{\lambda, \Omega}$ and $E_{\lambda}(L_{G+H})$ for three-edge updates of the cycle on 1000 vertices.}
{\begin{tabular}{@{}ccc@{}}
\toprule
\(\operatorname{diag}(\Omega)\)
& \(\dim\mathcal C_{\lambda,\Omega}\)
& \(m_\lambda(L_{G+H})\) \\
\midrule
\((0.030000,0.040000,0.050000)\) & \(0\) & \(2\) \\
\((0.0264233461,0.0264233461,0.0264233461)\) & \(1\) & \(3\) \\
\((0.0521230248,0.0521230248,0.0521230248)\) & \(1\) & \(3\) \\
\((0.0264233461,0.0691772181,0.0264233461)\) & \(2\) & \(4\) \\
\bottomrule
\end{tabular}}
\tabnote{The original cycle has unit edge weights. The updated edges are
\(\{1,501\}\), \(\{101,601\}\), and \(\{201,701\}\), in the order of the
diagonal entries of \(\Omega\). The eigenvalue is
\(\lambda=2-2\cos(8\pi/1000)\). The second column gives the dimension of
$\mathcal{C}_{\lambda, \Omega}$; the third gives the multiplicity of
\(\lambda\) in the updated Laplacian. The inherited subspace has dimension two.
The exceptional weights in the last three rows are shown rounded;
the stated dimensions refer to the corresponding exact weights.}
\label{tab:large-cycle-weights}
\end{table}

\end{example}

\section*{Disclosure statement}
No potential conflict of interest was reported by the author(s).

\section*{Funding}
The work of the second author is funded by  FCT - Fundação para a Ciência e a Tecnologia, I.P., through national funds, under the project UID/04561/2025. https://doi.org/10.54499/UID/04561/2025

\section*{Declaration of  use of generative AI and AI-assisted technologies
}



During the preparation of this work, the authors used ChatGPT, developed
by OpenAI, to assist with improving the clarity of the presentation,
checking notation and internal consistency, identifying possible gaps or
errors in mathematical arguments, suggesting alternative formulations of
statements and proofs, and identifying potentially relevant literature.
All the AI suggestions were verified by the authors independently. 
ChatGPT also assisted with preparing Python code to verify the
calculations in the examples. 

The authors have confirmed the originality and accuracy of the
manuscript. They have also reviewed the terms of use applicable to
ChatGPT and confirm that the AI-assisted content is suitable for
publication in accordance with those terms.

\appendix
\section{List of notation}
\label{app:notation}

Below, we present a summary of the notation used more frequently in the paper. References point to
its definition in the main text; for unnumbered definitions, the location
is given relative to a labelled equation or result. For standard symbols
not explicitly defined in the text, a representative use is indicated.

\begingroup
\small
\setlength{\tabcolsep}{5pt}
\renewcommand{\arraystretch}{1.15}
\setlength{\LTpre}{0.7\baselineskip}
\setlength{\LTpost}{0.7\baselineskip}

\begin{longtable}{@{}
  >{\raggedright\arraybackslash}p{0.22\textwidth}
  >{\raggedright\arraybackslash}p{\dimexpr0.53\textwidth-4\tabcolsep\relax}
  >{\raggedright\arraybackslash}p{0.25\textwidth}
  @{}}

\toprule
\textbf{Notation} & \textbf{Description} & \textbf{Reference} \\
\midrule
\endfirsthead

\multicolumn{3}{@{}l}{\textit{List of notation (continued)}} \\[3pt]
\toprule
\textbf{Notation} & \textbf{Description} & \textbf{Reference} \\
\midrule
\endhead

\midrule
\multicolumn{3}{r@{}}{\textit{Continued on the next page}} \\
\endfoot

\bottomrule
\endlastfoot

\(\sigma(A)\)
& Spectrum of a square matrix \(A\).
& After \eqref{eq:eigenspace-notation} \\

\(E_\mu(A)\)
&  \(\ker(A-\mu I)\); the \(\mu\)-eigenspace when
  \(\mu\in\sigma(A)\).
& \eqref{eq:eigenspace-notation} \\

\(m_\mu(A)\)
& Dimension of \(E_\mu(A)\); the geometric multiplicity of
  \(\mu\) when \(\mu\in\sigma(A)\).
& \eqref{eq:eigenspace-notation} \\

\(G=(V,E,\eta)\)
& Original positively weighted graph, with vertex set
  \(V=\{1,\ldots,n\}\), edge set \(E\), and weight function \(\eta\).
& Subsection~\ref{sec:positivegraph} \\

\(L_G\)
& Laplacian of the original graph $G$.
& \eqref{eq:weighted-laplacian-sum} \\

\(\lambda\)& Eigenvalue of the original Laplacian \(L_G\) under consideration.
& Before \eqref{eq:spectral-projectors} \\

\(m\)
&  \(m_\lambda(L_G)\).
& Before \eqref{eq:spectral-projectors} \\

\(Q_\lambda\)
& Matrix with orthonormal columns spanning \(E_\lambda(L_G)\subseteq \mathbb{R}^{n}\).
& Before \eqref{eq:spectral-projectors} \\

\(P_\lambda\)
& Orthogonal projector \(Q_\lambda Q_\lambda^T\) onto
  \(E_\lambda(L_G)\).
& \eqref{eq:spectral-projectors} \\

\(P_{\lambda,\perp}\)
& Orthogonal projector \(I_n-P_\lambda\) onto
  \(E_\lambda(L_G)^\perp\).
& \eqref{eq:spectral-projectors} \\

\(H=(V,F,\omega)\)
& Update graph on \(V\), with updated edge set \(F\) and
  positive weight-increment function \(\omega\).
& Before \eqref{eq:Omega-definition} \\

\(p\)
& Number of updated edges: \(p=|F|\).
& Before \eqref{eq:Omega-definition} \\

\(L_{G+H}\)
& Laplacian of the updated graph $G+H$.
& \eqref{eq:updated-laplacian} \\

\(B\)
& Oriented incidence matrix of the update graph, with columns
  \(b_1,\ldots,b_p\).
& \eqref{eq:B-X-definitions} \\

\(\Omega\)
& Positive diagonal matrix of the update-edge weight increments.
& \eqref{eq:Omega-definition} \\

\(\omega_i\)
& \(i\)th diagonal entry of \(\Omega\);
  \((\omega_1,\ldots,\omega_p)\) is the associated vector of weights.
& Before \eqref{def:plambda} \\

\(X_\Omega\)
& Weighted incidence matrix of the update graph:
  \(X_\Omega=B\Omega^{1/2}\).
& \eqref{eq:B-X-definitions} \\

\(\mathcal U_\lambda\)
& Inherited eigenspace:
  \(E_\lambda(L_G)\cap E_\lambda(L_{G+H})\).
& \eqref{def:Ulambda} \\

\(\mathcal C_{\lambda,\Omega}\)
& Canonical emergent subspace:
  \(E_\lambda(L_{G+H})\cap\mathcal U_\lambda^\perp\).
& \eqref{def:ClambdaOmega} \\

\(\mathbin{\oplus^\perp}\)
& Orthogonal direct sum of subspaces.
& Used in \eqref{eq:inherited-emergent-decomposition} \\

\(G_\lambda\)
& Moore--Penrose inverse of \(L_G-\lambda I_n\).
& \eqref{def:Glambda} \\

\(\Gamma_\lambda\)
& Symmetric matrix \(B^TG_\lambda B\in \mathbb{R}^{p\times p}\).
& \eqref{def:gamma} \\

\(\mathcal K_\lambda\)
& Subspace \(\ker(Q_\lambda^TB)\) of \(\mathbb R^p\).
& \eqref{def:Klambda} \\

\(d\)
& Dimension of \(\mathcal K_\lambda\):
  \(p-\operatorname{rank}(B^TQ_\lambda)\).
& After \eqref{def:Klambda} \\

\(Z\)
& Matrix in \(\mathbb R^{p\times d}\) with orthonormal columns
  spanning \(\mathcal K_\lambda\).
& After \eqref{def:Klambda} \\

\(F_\lambda(\Omega)\)
& Reduced symmetric matrix
  \(Z^T(\Omega^{-1}+\Gamma_\lambda)Z\in \mathbb{R}^{d\times d}\).
& \eqref{eq:weight-matrix-pencil} \\

\(M_\lambda\)
& Weight-independent symmetric matrix \(Z^T\Gamma_\lambda Z\).
& \eqref{def:Hlambda} \\

\(\widehat F_\lambda(s)\)
& Reduced matrix expressed in reciprocal variables:
  \(\widehat F_\lambda(s)
    =M_\lambda+\sum_{i=1}^p s_iZ^Te_ie_i^TZ\).
& \eqref{widehatF} \\

\(s,\ s_i\)
& Reciprocal-weight vector \(s=(s_1,\ldots,s_p)\), where
  \(s_i=1/\omega_i\).
& Before \eqref{def:plambda} \\

\(p_\lambda(s)\)
& Determinant \(\det F_\lambda(\Omega)\), viewed as a polynomial
  in the reciprocal weights.
& \eqref{def:plambda} \\

\(T_{\lambda,\Omega}\)
& Edge-difference map \(y\mapsto X_\Omega^Ty\) from
  \(E_\lambda(L_{G+H})\) to \(\mathbb R^p\).
& \eqref{eq:inheritance-edge-difference-map} \\

\(\mathcal S_{\lambda,\Omega}\)
& Edge-difference space
  \(\operatorname{range}(T_{\lambda,\Omega})\).
& \eqref{def:SlambdaOmega} \\

\(\Phi_{\lambda,\Omega}\)
& Inverse of the restriction of \(T_{\lambda,\Omega}\) to
  \(\mathcal C_{\lambda,\Omega}\); maps edge differences to
  their canonical emergent representatives.
& \eqref{def:Phi-lambda-Omega} \\

\(\mathcal Z_{\lambda,\Omega}\)
& Rescaled edge-difference space
  \(\{\Omega B^Ty:y\in E_\lambda(L_{G+H})\}\), equal to
  \(\Omega^{1/2}\mathcal S_{\lambda,\Omega}\).
& \eqref{def:ZlambdaOmega} \\

\end{longtable}
\endgroup

\end{document}